\documentclass[11pt,%
    reqno
]{amsart}
   
\usepackage{amsmath, amssymb, comment, graphicx, url, amsthm, multicol, latexsym, enumerate, bbm, mathtools, physics, microtype, cite, xcolor, float, units, wrapfig, caption, subcaption}
\usepackage[all,cmtip]{xy}
\usepackage[hmargin=1in, vmargin=1in]{geometry} 
\usepackage{tikz}
\usepackage{tikz-cd}  
    \usetikzlibrary{shapes.geometric, arrows}
\usepackage[cm]{sfmath}
\usepackage{hyperref}
    \hypersetup{%
        colorlinks=true,
        linkcolor=blue,
        citecolor=magenta,
        filecolor=magenta,      
        urlcolor=magenta,
        bookmarksopen=true,   
    }%
\usepackage{soul}
\usepackage{enumitem}
\usepackage[capitalize,noabbrev]{cleveref}   
    \crefname{const}{Construction}{Constructions}
    \crefname{subsection}{Subsection}{Subsections}   
    \crefformat{ineq}{Inequality~(#2#1#3)}   
        \crefrangeformat{ineq}{Inequalities~(#3#1#4) to~(#5#2#6)}
        \crefmultiformat{ineq}{Inequalities~(#2#1#3)}%
            { and~(#2#1#3)}{, (#2#1#3)}{ and~(#2#1#3)}
    \crefname{subthm}{Theorem}{Theorems}   
    \crefname{sublem}{Lemma}{Lemmas}   
    \crefname{subprop}{Proposition}{Propositions}   
    \crefname{subcor}{Corollary}{Corollaries}   

\usepackage[colorinlistoftodos]{todonotes}
\definecolor{andresblue}{rgb}{0,0.72,0.92}
\definecolor{andrespink}{rgb}{1,0,1}

\definecolor{munsell}{rgb}{0.0, 0.5, 0.69}

\newcommand{\TM}[1]{\todo[size=\footnotesize,inline,color=orange!50]{\sf #1 \mbox{} \hfill --- Tyrrell}\noindent}

\newtheorem{theorem}{Theorem}[section]
\newtheorem{proposition}[theorem]{Proposition}

\newtheorem{lemma}[theorem]{Lemma}
\newtheorem{corollary}[theorem]{Corollary}

\theoremstyle{definition}
\newtheorem{remark}[theorem]{Remark}
\newtheorem{definition}[theorem]{Definition}

\newtheorem{example}[theorem]{Example}

\newtheorem{question}[theorem]{Question}
\theoremstyle{remark}

\crefname{example}{Example}{Examples}
\crefname{theorem}{Theorem}{Theorems}
\crefname{lem}{Lemma}{Lemmas}
\crefname{prop}{Proposition}{Propositions}
\crefname{figure}{Figure}{Figures}
\crefname{fig}{Figure}{Figures}
\crefname{remark}{Remark}{Remarks}
\crefname{rem}{Remark}{Remarks}
\crefname{cor}{Corollary}{Corollaries}
\crefname{corollary}{Corollary}{Corollaries}
\crefname{conjecture}{Conjecture}{Conjectures}
\crefname{conj}{Conjecture}{Conjectures}
\crefname{ex}{Example}{Examples}

\renewcommand{\emptyset}{\varnothing}

\newcommand{\B}{\mathcal{B}}
\newcommand{\R}{\mathbb{R}}
\newcommand{\Z}{\mathbb{Z}}

\newcommand{\ee}{\mathbf{e}}
\newcommand{\bx}{\mathbf{b}}   
\newcommand{\bb}{\mathbf{b}}   
\newcommand{\bv}{\mathbf{v}}   
\newcommand{\bw}{\mathbf{w}}   
\newcommand{\I}{\mathcal{I}}   

\newcommand{\Tan}[1]{\mathcal{T}_{#1}}   
\newcommand{\Y}{\mathcal{Y}}
\newcommand{\cZ}{\mathcal{Z}}
\newcommand{\SSS}{\mathfrak{S}_{n}}   
\newcommand{\gpfp}{\mathfrak{X}}   
\newcommand{\Xn}{\gpfp_{n}(\bb)}   
\newcommand{\y}{\mathbf{v}}   
\newcommand{\adjto}{\sim}   
\newcommand{\Xnb}{\overline{\gpfp}_{n}(\bb)}   
\newcommand{\Xnp}{\gpfp'_{n}(\bb)}   
\newcommand{\PF}{\mathsf{PF}}
\newcommand{\St}{\mathsf{St}}

\newcommand{\des}{\mathsf{des}}
\newcommand{\Des}{\mathsf{Des}}

\newcommand{\BF}{\mathsf{BF}} 
\newcommand{\defterm}[1]{\emph{#1}}
\newcommand*{\defing}{\defterm}
\newcommand{\fx}{f_{\, \bb}}   

\DeclareMathOperator{\conv}{conv}

\DeclareMathOperator{\Ver}{Vert}   

\makeatletter
\DeclareRobustCommand{\rvdots}{%
    \vbox{%
        \baselineskip4\p@\lineskiplimit\z@
        \kern-\p@
        \hbox{.}\hbox{.}\hbox{.}
    }%
}%
\makeatother

\makeatletter 
\newtheorem*{rep@theorem}{\rep@title}\newcommand{\newreptheorem}[2]{%
\newenvironment{rep#1}[1]{%
\def\rep@title{\bf #2 \ref{##1}}%
\begin{rep@theorem}}%
{\end{rep@theorem}}}
\makeatother
\newreptheorem{theorem}{Theorem}

\DeclarePairedDelimiter{\verts}{\lvert}{\rvert}

\newcommand*{\parens}[1]{\left\lparen #1 \right\rparen}

\newcommand*{\bigbraces}[1]{\bigl\lbrace #1 \bigr\rbrace}

\newcommand*{\Biggbraces}[1]{\Biggl\lbrace #1 \Biggr\rbrace}

\newcommand{\card}{\verts}
\newcommand{\setof}{\braces\{\}}

\newcommand*{\maps}{\colon}
\newcommand*{\deftobe}{\coloneqq}   
\newcommand*{\sst}{\,:\,}   

\begin{document}


\title{Combinatorics of generalized parking-function polytopes}


\author{Margaret M. Bayer}
\author{Steffen Borgwardt}
\author{Teressa Chambers}
\author{Spencer Daugherty}
\author{Aleyah Dawkins}
\author{Danai Deligeorgaki}
\author{Hsin-Chieh Liao}
\author{Tyrrell McAllister}
\author{Angela Morrison}
\author{Garrett Nelson}
\author{Andr\'es R. Vindas-Mel\'endez}


\begin{abstract}

For $\bb=(b_1,\dots,b_n)\in \mathbb{Z}_{>0}^n$, a $\bb$\emph{-parking function} is defined to be a sequence $(\beta_1,\dots,\beta_n)$ of positive integers whose nondecreasing rearrangement $\beta'_1\leq \beta'_2\leq \cdots \leq \beta'_n$ satisfies $\beta'_i\leq b_1+\cdots + b_i$.
The $\bb$-parking-function polytope $\Xn$ is the convex hull of all $\bb$-parking functions of length $n$ in $\mathbb{R}^n$.  
Geometric properties of $\Xn$ were previously explored in the specific case where $\bb=(a,b,b,\dots,b)$ and were shown to generalize those of the classical parking-function polytope.
In this work, we study $\Xn$ in full generality. 
We present a minimal inequality and vertex description for $\Xn$, prove it is a generalized permutahedron, and study its $h$-polynomial.
Furthermore, we investigate $\Xn$ through the perspectives of building sets and polymatroids, allowing us to identify its combinatorial types and obtain bounds on its combinatorial and circuit diameters.

\end{abstract}


\maketitle

\section{Introduction}

A \textit{classical parking function} of length $n$ is a list $(\alpha_1, \alpha_2, \ldots, \alpha_n)$ of positive integers whose nondecreasing rearrangement $\alpha'_1 \leq \alpha'_2 \leq \cdots \leq \alpha'_n$ satisfies $\alpha'_i \leq i$. 
It is well-known that the number of classical parking functions of length $n$ is $(n+1)^{n-1}$.
This number surfaces in a variety of places; for example, it counts the number of labeled rooted forests on $n$ vertices and the number of regions of a Shi arrangement (see \cite{Bon} for further discussion).
Let $\PF_n$ denote the convex hull in $\R^n$ of all classical parking functions of length $n$.
In 2020, Stanley \cite{Sta} asked for the volume and the number of vertices, faces, and lattice points of $\PF_n$. 
These questions were first answered by Amanbayeva and Wang in \cite{AW}.
In this article, we present results on the combinatorial and geometric properties of the convex hull of a generalization of classical parking functions, namely $\bb$-parking functions.

\begin{definition}
Let $\bb=(b_1,\dots,b_n)\in \Z_{>0}^n$.  
A $\bb$\emph{-parking function} is a list $(\beta_1,\dots,\beta_n)$ of positive integers whose nondecreasing rearrangement $\beta'_1\leq \beta'_2\leq \cdots \leq \beta'_n$ satisfies $\beta'_i\leq b_1+\cdots + b_i$.
\end{definition}

This generalization of parking functions has been previously explored from an enumerative perspective by Yan \cite{Yan2, Yan, Bon} and Pitman and Stanley \cite{PitmanStanley}.
In~\cite{HanadaLentferVindas}, Hanada, Lentfer, and Vindas-Mel\'endez generalized the classical parking-function polytope to the \emph{$\bb$-parking-function polytope} $\Xn$, i.e., the convex hull of all $\bb$-parking functions of length $n$ in $\R^n$ for $\bb=(b_1,\dots,b_n)\in \Z_{>0}^n$.
The work of Hanada, Lentfer, and Vindas-Mel\'endez focused on the $\bb$-parking-function polytope in the special case when $\bb = (a,b,b,\ldots, b)$.
In particular, they described the face structure and provided a closed volume formula for the $\bb$-parking-function polytope in that special case.  
These parking-function polytopes have surprisingly arisen in different guises and in connection to other polytopes and combinatorial objects, including partial permutahedra \cite{BCC,HS}, Pitman-Stanley polytopes \cite{HanadaLentferVindas,PitmanStanley}, polytopes of win vectors \cite{Backman, BartelsMountWelsh}, and stochastic sandpile models \cite{Selig}.

In this paper, we explore $\bb$-parking-function polytopes for general $\bb$. 
Our contributions about $\bb$-parking-function polytopes lend themselves to fruitful exploration from different perspectives and propose connections to seemingly unrelated combinatorial objects. 
We present definitions, terminology, and references in each relevant section. 
The paper is structured as follows: 
\begin{itemize}[leftmargin=*]\setlength\itemsep{0.5em}

    \item 
    We begin \cref{sec:polyhedra_perspectives} by studying the $\bb$-parking-function polytope from a polyhedral perspective.  
    We show that $\Xn$ is (apart from a single degenerate case) an $n$-dimensional simple polytope (\cref{prop: Xn is n-dimensional,subthm: Xn is simple}), and we give explicit descriptions of the vertices, facets, and edges of $\Xn$ (\cref{subthm: vertices of Xn,subthm: facet-defining inequalities,prop:edges}, respectively).  
    We then show that, up to a unimodular linear isomorphism, $\Xn$ is a generalized permutahedron (\cref {thm: x-park_is_GPerm}).  
    We conclude the \lcnamecref{sec:polyhedra_perspectives} by presenting the
    $h$-polynomial and $f$-vector of $\Xn$
    (\cref{thm:h-poly_of_Xn,cor: f-vector of Xn}, respectively).

    \item
    In \cref{sec:bulding_set} we take a building-set perspective to study $\Xn$, which allows us to completely characterize the face structure of $\Xn$ for any $\bb\in\Z_{>0}^n$ (\cref{thm:face_structure_Xn}).  
    As a consequence, we show that, up to combinatorial equivalence, every $\bb$-parking-function polytope is either the classical parking-function polytope $\PF_n$ or the stellohedron $\St_n$ (\cref
    {cor:combinatorial_equivalence}).

    \item
    We identify $\Xn$ as a special polymatroid (\cref{thm:polymatroid}) in \cref {sec:polymatroid_perspective} and use this perspective to derive linear bounds on its combinatorial and circuit diameters (\cref {thm:boundtight} and \cref {thm:circuitdiameter}, respectively). 

    \item
    In \cref{sec:revisit}, we revisit the classical parking-function polytope $\PF_n$ and summarize its known and new connections to other established families of polytopes. 
    We conclude by adding a description of $\PF_n$ as a projection of a relaxed Birkhoff polytope (\cref{thm:birkhoff}), and $\Xn$ as a projection of a relaxed-partition polytope (\cref{cor:birkhoff}). 

\end{itemize}


\section{Polyhedral perspectives on $\Xn$}
\label{sec:polyhedra_perspectives}

Fix, throughout this work, a positive integer $n$ and a vector $\bb = (b_1, \dots, b_n)\in \Z_{> 0}^n$.  
The following definition introduces our main object of study, i.e., a polytope associated to $\bb$-parking functions, which we denote $\Xn$.  
Standard references on polytopes include \cite{Grunbaum} and \cite{Ziegler}.  

\begin{definition}
    The \emph{$\bb$-parking-function polytope} $\Xn$ is the convex hull of all $\bb$-parking functions in $\R^n$.
\end{definition} 

Recall that the linear action of the symmetric group $\SSS$ on $\R^{n}$ is defined by $\pi\parens{\sum_{i=1}^{n} a_{i} \, \ee_{i}} \deftobe \sum_{i=1}^{n} a_{i} \, \ee_{\pi(i)}$ for all $\pi \in \SSS$, where $\ee_{1}, \dotsc, \ee_{n}$ is the standard basis of $\R^{n}$.  
Note that by definition, the polytope $\Xn$ is symmetric under the action of $\SSS$.  
 For convenience, we use the notation \[{S_{i} \deftobe \sum_{j = 1}^{i} b_{j}} \text{ for } 1 \le i \le n \text{\; and \;} \y_{k} \deftobe (1, \dotsc, 1, S_{k+1}, S_{k+2}, \dotsc, S_{n}) \in \R^{n} \text{ for } 0 \le k \le n.\]  
Further observe that the~$\y_{k}$ are $\bb$-parking functions, so $\y_{k} \in \Xn$ for all $k$.  
We will see below that the permutations $\pi(\y_{k})$ of the~$\y_{k}$ are precisely the vertices of $\Xn$ (\cref{subthm: vertices of Xn}).

A theme that will appear early in our study of $\Xn$ is that the case in which $b_1 = 1$ behaves differently from when $b_1 \ge 2$.
Indeed, observe that, if $b_{1} \ge 2$, or if $b_{1} = 1$ and $k \ge 2$, then the index~$k$ in $\y_{k}$ is precisely the number of entries in $\y_{k}$ that equal $1$.  
In particular, when $b_{1} \ge 2$, the $\bb$-parking functions $\y_{0}, \y_{1}, \dotsc, \y_{n}$ are distinct.
However, if $b_{1} = 1$, then $\y_{0} = \y_{1} = (1, S_{2}, S_{3},\dotsc, S_{n})$. 
This trivial coincidence has important consequences for the structure of the polytope $\Xn$.
\cref{b123} shows the $\bb$-parking-function polytopes $\gpfp_3(1,2,3)$ (left) and
$\gpfp_3(2,3,4)$ (right).  
Note also that many lattice points in these polytopes are not parking functions. 
For example, $(2,2,2)$ is in $\gpfp_3(1,2,3)$, but is not a $\bb$-parking function for $\bb=(1,2,3)$.

\begin{figure}[ht]
\begin{center}
\scalebox{.5}{
\begin{tikzpicture}%
	[x={(-0.587802cm, -0.404450cm)},
	y={(0.809005cm, -0.293947cm)},
	z={(0.000078cm, 0.866034cm)},
	scale=1.000000,
	back/.style={loosely dashed, thick},
	edge/.style={color=black, thick},
	facet/.style={fill=white,fill opacity=.10000},
	vertex/.style={inner sep=0cm, draw=white,fill=white }]

\coordinate (1.00000, 1.00000, 1.00000) at (1.00000, 1.00000, 1.00000);
\coordinate (1.00000, 1.00000, 6.00000) at (1.00000, 1.00000, 6.00000);
\coordinate (1.00000, 3.00000, 6.00000) at (1.00000, 3.00000, 6.00000);
\coordinate (1.00000, 6.00000, 1.00000) at (1.00000, 6.00000, 1.00000);
\coordinate (1.00000, 6.00000, 3.00000) at (1.00000, 6.00000, 3.00000);
\coordinate (3.00000, 1.00000, 6.00000) at (3.00000, 1.00000, 6.00000);
\coordinate (3.00000, 6.00000, 1.00000) at (3.00000, 6.00000, 1.00000);
\coordinate (6.00000, 1.00000, 1.00000) at (6.00000, 1.00000, 1.00000);
\coordinate (6.00000, 1.00000, 3.00000) at (6.00000, 1.00000, 3.00000);
\coordinate (6.00000, 3.00000, 1.00000) at (6.00000, 3.00000, 1.00000);
\draw[edge,back] (1.00000, 1.00000, 1.00000) -- (1.00000, 1.00000, 6.00000);
\draw[edge,back] (1.00000, 1.00000, 1.00000) -- (1.00000, 6.00000, 1.00000);
\draw[edge,back] (1.00000, 1.00000, 1.00000) -- (6.00000, 1.00000, 1.00000);
\node[vertex] at (1.00000, 1.00000, 1.00000)     {\tiny $(1,1,1)$};
\fill[facet] (6.00000, 3.00000, 1.00000) -- (3.00000, 6.00000, 1.00000) -- (1.00000, 6.00000, 3.00000) -- (1.00000, 3.00000, 6.00000) -- (3.00000, 1.00000, 6.00000) -- (6.00000, 1.00000, 3.00000) -- cycle {};
\fill[facet] (6.00000, 3.00000, 1.00000) -- (6.00000, 1.00000, 1.00000) -- (6.00000, 1.00000, 3.00000) -- cycle {};
\fill[facet] (3.00000, 6.00000, 1.00000) -- (1.00000, 6.00000, 1.00000) -- (1.00000, 6.00000, 3.00000) -- cycle {};
\fill[facet] (3.00000, 1.00000, 6.00000) -- (1.00000, 1.00000, 6.00000) -- (1.00000, 3.00000, 6.00000) -- cycle {};
\draw[edge] (1.00000, 1.00000, 6.00000) -- (1.00000, 3.00000, 6.00000);
\draw[edge] (1.00000, 1.00000, 6.00000) -- (3.00000, 1.00000, 6.00000);
\draw[edge] (1.00000, 3.00000, 6.00000) -- (1.00000, 6.00000, 3.00000);
\draw[edge] (1.00000, 3.00000, 6.00000) -- (3.00000, 1.00000, 6.00000);
\draw[edge] (1.00000, 6.00000, 1.00000) -- (1.00000, 6.00000, 3.00000);
\draw[edge] (1.00000, 6.00000, 1.00000) -- (3.00000, 6.00000, 1.00000);
\draw[edge] (1.00000, 6.00000, 3.00000) -- (3.00000, 6.00000, 1.00000);
\draw[edge] (3.00000, 1.00000, 6.00000) -- (6.00000, 1.00000, 3.00000);
\draw[edge] (3.00000, 6.00000, 1.00000) -- (6.00000, 3.00000, 1.00000);
\draw[edge] (6.00000, 1.00000, 1.00000) -- (6.00000, 1.00000, 3.00000);
\draw[edge] (6.00000, 1.00000, 1.00000) -- (6.00000, 3.00000, 1.00000);
\draw[edge] (6.00000, 1.00000, 3.00000) -- (6.00000, 3.00000, 1.00000);
\node[vertex] at (1.00000, 1.00000, 6.00000)     {\tiny $(1,1,6)$};
\node[vertex] at (1.00000, 3.00000, 6.00000)     {\tiny $(1,3,6)$};
\node[vertex] at (1.00000, 6.00000, 1.00000)     {\tiny $(1,6,1)$};
\node[vertex] at (1.00000, 6.00000, 3.00000)     {\tiny $(1,6,3)$};
\node[vertex] at (3.00000, 1.00000, 6.00000)     {\tiny $(3,1,6)$};
\node[vertex] at (3.00000, 6.00000, 1.00000)     {\tiny $(3,6,1)$};
\node[vertex] at (6.00000, 1.00000, 1.00000)     {\tiny $(6,1,1)$};
\node[vertex] at (6.00000, 1.00000, 3.00000)     {\tiny $(6,1,3)$};
\node[vertex] at (6.00000, 3.00000, 1.00000)     {\tiny $(6,3,1)$};
\end{tikzpicture}

\qquad\qquad

\begin{tikzpicture}%
	[x={(-0.587802cm, -0.404450cm)},
	y={(0.809005cm, -0.293947cm)},
	z={(0.000078cm, 0.866034cm)},
	scale=1.000000,
	back/.style={loosely dashed, thick},
	edge/.style={color=black, thick},
	facet/.style={fill=white, fill opacity=0.100000},
	vertex/.style={inner sep=0cm, draw=white,fill=white }]

\coordinate (1.00000, 1.00000, 1.00000) at (1.00000, 1.00000, 1.00000);
\coordinate (1.00000, 1.00000, 9.00000) at (1.00000, 1.00000, 9.00000);
\coordinate (1.00000, 5.00000, 9.00000) at (1.00000, 5.00000, 9.00000);
\coordinate (1.00000, 9.00000, 1.00000) at (1.00000, 9.00000, 1.00000);
\coordinate (1.00000, 9.00000, 5.00000) at (1.00000, 9.00000, 5.00000);
\coordinate (2.00000, 5.00000, 9.00000) at (2.00000, 5.00000, 9.00000);
\coordinate (2.00000, 9.00000, 5.00000) at (2.00000, 9.00000, 5.00000);
\coordinate (5.00000, 1.00000, 9.00000) at (5.00000, 1.00000, 9.00000);
\coordinate (5.00000, 2.00000, 9.00000) at (5.00000, 2.00000, 9.00000);
\coordinate (5.00000, 9.00000, 1.00000) at (5.00000, 9.00000, 1.00000);
\coordinate (5.00000, 9.00000, 2.00000) at (5.00000, 9.00000, 2.00000);
\coordinate (9.00000, 1.00000, 1.00000) at (9.00000, 1.00000, 1.00000);
\coordinate (9.00000, 1.00000, 5.00000) at (9.00000, 1.00000, 5.00000);
\coordinate (9.00000, 2.00000, 5.00000) at (9.00000, 2.00000, 5.00000);
\coordinate (9.00000, 5.00000, 1.00000) at (9.00000, 5.00000, 1.00000);
\coordinate (9.00000, 5.00000, 2.00000) at (9.00000, 5.00000, 2.00000);
\draw[edge,back] (1.00000, 1.00000, 1.00000) -- (1.00000, 1.00000, 9.00000);
\draw[edge,back] (1.00000, 1.00000, 1.00000) -- (1.00000, 9.00000, 1.00000);
\draw[edge,back] (1.00000, 1.00000, 1.00000) -- (9.00000, 1.00000, 1.00000);
\node[vertex] at (1.00000, 1.00000, 1.00000)     {\tiny $(1,1,1)$};
\fill[facet] (9.00000, 5.00000, 2.00000) -- (5.00000, 9.00000, 2.00000) -- (2.00000, 9.00000, 5.00000) -- (2.00000, 5.00000, 9.00000) -- (5.00000, 2.00000, 9.00000) -- (9.00000, 2.00000, 5.00000) -- cycle {};
\fill[facet] (9.00000, 5.00000, 2.00000) -- (5.00000, 9.00000, 2.00000) -- (5.00000, 9.00000, 1.00000) -- (9.00000, 5.00000, 1.00000) -- cycle {};
\fill[facet] (9.00000, 2.00000, 5.00000) -- (5.00000, 2.00000, 9.00000) -- (5.00000, 1.00000, 9.00000) -- (9.00000, 1.00000, 5.00000) -- cycle {};
\fill[facet] (9.00000, 5.00000, 2.00000) -- (9.00000, 2.00000, 5.00000) -- (9.00000, 1.00000, 5.00000) -- (9.00000, 1.00000, 1.00000) -- (9.00000, 5.00000, 1.00000) -- cycle {};
\fill[facet] (2.00000, 9.00000, 5.00000) -- (1.00000, 9.00000, 5.00000) -- (1.00000, 5.00000, 9.00000) -- (2.00000, 5.00000, 9.00000) -- cycle {};
\fill[facet] (5.00000, 9.00000, 2.00000) -- (2.00000, 9.00000, 5.00000) -- (1.00000, 9.00000, 5.00000) -- (1.00000, 9.00000, 1.00000) -- (5.00000, 9.00000, 1.00000) -- cycle {};
\fill[facet] (5.00000, 2.00000, 9.00000) -- (2.00000, 5.00000, 9.00000) -- (1.00000, 5.00000, 9.00000) -- (1.00000, 1.00000, 9.00000) -- (5.00000, 1.00000, 9.00000) -- cycle {};
\draw[edge] (1.00000, 1.00000, 9.00000) -- (1.00000, 5.00000, 9.00000);
\draw[edge] (1.00000, 1.00000, 9.00000) -- (5.00000, 1.00000, 9.00000);
\draw[edge] (1.00000, 5.00000, 9.00000) -- (1.00000, 9.00000, 5.00000);
\draw[edge] (1.00000, 5.00000, 9.00000) -- (2.00000, 5.00000, 9.00000);
\draw[edge] (1.00000, 9.00000, 1.00000) -- (1.00000, 9.00000, 5.00000);
\draw[edge] (1.00000, 9.00000, 1.00000) -- (5.00000, 9.00000, 1.00000);
\draw[edge] (1.00000, 9.00000, 5.00000) -- (2.00000, 9.00000, 5.00000);
\draw[edge] (2.00000, 5.00000, 9.00000) -- (2.00000, 9.00000, 5.00000);
\draw[edge] (2.00000, 5.00000, 9.00000) -- (5.00000, 2.00000, 9.00000);
\draw[edge] (2.00000, 9.00000, 5.00000) -- (5.00000, 9.00000, 2.00000);
\draw[edge] (5.00000, 1.00000, 9.00000) -- (5.00000, 2.00000, 9.00000);
\draw[edge] (5.00000, 1.00000, 9.00000) -- (9.00000, 1.00000, 5.00000);
\draw[edge] (5.00000, 2.00000, 9.00000) -- (9.00000, 2.00000, 5.00000);
\draw[edge] (5.00000, 9.00000, 1.00000) -- (5.00000, 9.00000, 2.00000);
\draw[edge] (5.00000, 9.00000, 1.00000) -- (9.00000, 5.00000, 1.00000);
\draw[edge] (5.00000, 9.00000, 2.00000) -- (9.00000, 5.00000, 2.00000);
\draw[edge] (9.00000, 1.00000, 1.00000) -- (9.00000, 1.00000, 5.00000);
\draw[edge] (9.00000, 1.00000, 1.00000) -- (9.00000, 5.00000, 1.00000);
\draw[edge] (9.00000, 1.00000, 5.00000) -- (9.00000, 2.00000, 5.00000);
\draw[edge] (9.00000, 2.00000, 5.00000) -- (9.00000, 5.00000, 2.00000);
\draw[edge] (9.00000, 5.00000, 1.00000) -- (9.00000, 5.00000, 2.00000);
\node[vertex] at (1.00000, 1.00000, 9.00000)     {\tiny $(1,1,9)$};
\node[vertex] at (1.00000, 5.00000, 9.00000)     {\tiny $(1,5,9)$};
\node[vertex] at (1.00000, 9.00000, 1.00000)     {\tiny $(1,9,1)$};
\node[vertex] at (1.00000, 9.00000, 5.00000)     {\tiny $(1,9,5)$};
\node[vertex] at (2.00000, 5.00000, 9.00000)     {\tiny $(2,5,9)$};
\node[vertex] at (2.00000, 9.00000, 5.00000)     {\tiny $(2,9,5)$};
\node[vertex] at (5.00000, 1.00000, 9.00000)     {\tiny $(5,1,9)$};
\node[vertex] at (5.00000, 2.00000, 9.00000)     {\tiny $(5,2,9)$};
\node[vertex] at (5.00000, 9.00000, 1.00000)     {\tiny $(5,9,1)$};
\node[vertex] at (5.00000, 9.00000, 2.00000)     {\tiny $(5,9,2)$};
\node[vertex] at (9.00000, 1.00000, 1.00000)     {\tiny $(9,1,1)$};
\node[vertex] at (9.00000, 1.00000, 5.00000)     {\tiny $(9,1,5)$};
\node[vertex] at (9.00000, 2.00000, 5.00000)     {\tiny $(9,2,5)$};
\node[vertex] at (9.00000, 5.00000, 1.00000)     {\tiny $(9,5,1)$};
\node[vertex] at (9.00000, 5.00000, 2.00000)     {\tiny $(9,5,2)$};
\end{tikzpicture}
}
\caption{The polytopes $\gpfp_3(1,2,3)$ (left) and $\gpfp_3(2,3,4)$ (right).}
\label{b123}
\end{center}
\end{figure}

We begin by proving that, with the exception of a single ``degenerate'' case, $\Xn$ is an $n$-dimensional polytope.

\begin{proposition}\label{prop: Xn is n-dimensional}
    For all $n$, $\dim(\Xn)=n$, unless $n = 1 = b_{1}$, in which case $\dim(\Xn)=0$.
\end{proposition}

\begin{proof}
    In the exceptional case $n = 1 = b_{1}$, the polytope $\Xn$ contains only the point~$(1)$ in $\R^{1}$.  
    In all other cases, $\Xn$ is $n$-dimensional because, if $b_{1} \ge 2$, then $\setof{\y_{k} \sst 0 \le k \le n}$ is a set of $n + 1$ affinely independent points in $\Xn$, and, if $b_1 = 1$ but $n \ge 2$, then $\setof{(S_2, 1, S_3, \dotsc, S_n)} \cup \setof{\y_k \sst 1 \le k \le n}$ is a set of $n + 1$ affinely independent points in $\Xn$.
\end{proof}

In view of \cref{prop: Xn is n-dimensional}, we exclude the degenerate case $n = 1 = b_{1}$ from this point forward.  
That is, we proceed throughout the rest of this paper with the additional assumption that either $n \ge 2$ or, if $n = 1$, then $b_{1} \ge 2$.  
Thus, $\dim \Xn = n$.


\subsection{Vertices, facets, and edges of \texorpdfstring{$\Xn$}{$\Xn$}}

In this subsection, we describe the vertices, facets, and edges of the $\bb$-parking-function~polytope $\Xn$.  
In particular, we find that $\Xn$ is a simple polytope, meaning that the number of edges of $\Xn$ that are incident to each vertex of $\Xn$ is $n$.

\begin{theorem} \label{thm:inequality_description}
    \mbox{}
    \begin{enumerate}[label=(\alph*), ref={\thetheorem(\alph*)}]
        \item
        \label[subthm]{subthm: Xn is simple}%
        The polytope $\Xn$ is simple.

        \item  
        \label[subthm]{subthm: vertices of Xn}%
        The set of vertices of $\Xn$ is $\setof{\pi(\y_{k}) \sst \text{$\pi \in \mathfrak{S}_n$ and $0 \le k \le n$}}$.
           
        \item  
        \label[subthm]{subthm: facet-defining inequalities}%
        The \emph{minimal} inequality description
        of $\Xn$ is the set of points $(x_1, \dotsc, x_n)$ in $\R^n$ such that $x_i \ge 1$ for $1 \le i \le n$ and
        \begin{align}
            \sum_{i \in I}
            x_i
            &\quad \le \quad 
            \sum_{i = n - \card{I} + 1}^{n} \, S_{i}, 
            \label[ineq]{x-parking}\tag{{\color{blue}{$\text{$\bb$-parking}$}}}
            \\
        \intertext{or, equivalently,}
            \sum_{i \in I}
            x_i
            &\quad \le \quad 
            \sum_{j=1}^n 
            \min \setof{j, \card{I}} \, b_{n-j+1},
            \notag 
        \end{align}
        for all nonempty $I \subseteq [n]$ if $b_{1} \ge 2$, and for all nonempty $I \subseteq [n]$ with $\card{I} \ne n-1$ if $b_{1} = 1$.
        In other words, these inequalities are precisely the facet-defining inequalities for $\Xn$.
    \end{enumerate}
\end{theorem} 

Before proceeding to the proof of \cref{thm:inequality_description}, we note that, even when $b_{1} = 1$, every point in $\Xn$ satisfies the \eqref{x-parking}~inequalities for \emph{all} subsets $I \subseteq [n]$, including those where $\card{I} = n-1$.  
However, when $b_{1}$ = 1, these ``$\card{I} = n-1$'' \eqref{x-parking}~inequalities are redundant because they follow from the \eqref{x-parking}~inequality for $[n]$ together with the inequalities $x_{i} \ge 1$. 
Let $m \in [n]$ and $I \deftobe [n] \setminus \setof{m}$.  
Then the \eqref{x-parking}~inequality for $[n]$ is
\begin{equation*}
    \sum_{i \in [n]} x_{n} \le \sum_{j=1}^{n-1} j \, b_{n-j + 1} + n.
\end{equation*}
Thus, subtracting the inequality $x_{m} \ge 1$ yields
\begin{equation*}
    \sum_{i \in I} x_{i}
    \le %
    \sum_{j=1}^{n-1} 
    j \, b_{n-j + 1} + (n-1)
    = %
    \sum_{j=1}^n 
    \min \setof{j, \card{I}} \, b_{n-j+1},
\end{equation*}
which is the \eqref{x-parking}~inequality for $I$.  
Nonetheless, none of the other inequalities can be eliminated, as the following proof demonstrates \emph{inter alia}.

We also remark here that \cref{thm:inequality_description} can be proved using the established theory of polymatroids. 
See \cref{sec:polymatroid_perspective} for further discussion of this connection. 
However, for completeness, we provide a proof here of \cref{thm:inequality_description} that does not rely on that theory.
In fact, it is the details of the proof of \cref{thm:inequality_description} that motivate the forthcoming connections with generalized permutahedra and polymatroids.  

\begin{proof}[Proof of \cref{thm:inequality_description}]
    We temporarily write $P(\bb)$ for the polytope of points in $\R^{n}$ satisfying the system of inequalities given in \cref{subthm: facet-defining inequalities}.  
    The outline of the proof is as follows.  
    We first prove that $\Xn \subseteq P(\bb)$.  
    We then prove that the vertices of $P(\bb)$ are all of the form $\pi(\y_{k})$ as in \cref{subthm: vertices of Xn}. 
    Thus, the vertices of $P(\bb)$ are $\bb$-parking functions, which proves the converse containment $P(\bb) \subseteq \Xn$, giving us $\Xn = P(\bb)$ and the ``$\subseteq$'' direction of \cref{subthm: vertices of Xn}.  
    In the course of proving that every vertex of $P(\bb)$ is of the form $\pi(\y_{k})$, we will have shown that each such vertex satisfies precisely~$n$ of the inequalities given in \cref{subthm: facet-defining inequalities}, which demonstrates that $\Xn$ is simple (\cref{subthm: Xn is simple}).  
    We will then show that every point of the form $\pi(\y_{k})$ is in fact a vertex of $\Xn$, completing the proof of \cref{subthm: vertices of Xn}.  
    Finally, we will see that each of the inequalities in \cref{subthm: facet-defining inequalities} is satisfied with equality by some point in $\Xn$.  
    This fact, together with the fact that every vertex of $\Xn$ satisfies exactly $n$ such inequalities, implies that the system of inequalities in \cref{subthm: facet-defining inequalities} is minimal and thus will complete the proof of
    \cref{thm:inequality_description}.
   
    To prove that $\Xn \subseteq P(\bb)$, it suffices to show that every vertex of $\Xn$ is in $P(\bb)$.  
    To this end, let $\mathbf{v} = (v_{1}, \dotsc, v_{n}) \in \R^{n}$ be a vertex of $\Xn$, and let $I \subseteq [n]$.
    Write $\mathbf{v}' = (v'_{1}, \dotsc, v'_{n})$ for the weakly increasing reordering of $\mathbf{v}$.  
    Thus, the final $\card{I}$ coordinates of $\mathbf{v}'$, namely $\smash{v'_{\smash{n - \card{I} + 1}}, v'_{\smash{n - \card{I} + 2}}, \dotsc, v'_{n}}$, are the $\card{I}$ weakly largest coordinates of $\mathbf{v}$.  
    Moreover, $\mathbf{v}$ is a $\bb$\nobreakdash-parking function (because the vertices of the convex hull of a finite set are elements of that set), so $1 \le v'_{i} \le \smash{\sum_{j = 1}^{i} b_{j}}$ for $1 \le i \le n$.  
    Hence, $v_{i} \ge 1$ for $1 \le i \le n$, and
    \begin{equation*}
        \sum_{i \in I} v_{i} 
        \; \le \; %
        \sum_{i = n - \card{I} + 1}^{n} v'_{i} 
        \; \le \; %
        \sum_{i = n - \card{I} + 1}^{n} \, \sum_{j = 1}^{i} b_{j}
        \; = \; %
        \sum_{j = 1}^{n} \, \sum_{i = \max \setof{j, \, n - \card{I} + 1}}^{n} b_{j}
        \; = \; %
        \sum_{j = 1}^{n} \min \setof{n - j + 1, \card{I}} \, b_{j}.
    \end{equation*}
    Reversing the order of summation yields that $\mathbf{v}$ satisfies \cref{x-parking}.  
    Thus, $\Xn \subseteq P(\bb)$.

    We now prove the converse containment.
    Let $\bw = (w_{1}, \dotsc, w_{n})$ be a vertex of $P(\bb)$.  
    It suffices to show that $\bw$ is a $\bb$-parking function.  
    Since $P(\bb)$ and the set of $\bb$-parking functions are both invariant under the action of $\SSS$ on $\R^{n}$, we may suppose that $w_{1} \le \dotsb \le w_{n}$. 
    We will show that $\bw = \y_{k}$ for some $k$, making $\bw$ evidently a $\bb$-parking function.
   
    Let $\I$ be the family of all nonempty subsets $I \subseteq [n]$ such that
    \begin{equation}\label{eq: w is a vertex}
        \sum_{i \in I}
        w_i
        =
        \sum_{i = n - \card{I} + 1}^{n} \, S_{i},
    \end{equation}
    and furthermore such that $\card{I} \ne n-1$ if $b_{1} = 1$.
    Thus, $\I$ indexes the set of inequalities of the form~\eqref{x-parking} that the vertex $\bw$ satisfies with equality.  
    Let $k$ be the number of coordinates of $\bw$ that equal~$1$.  
    Since $P(\bb)$ is a polytope in $\R^{n}$, and $\bw$ is a vertex of this polytope, $\bw$ satisfies with equality at least $n$ of the inequalities that define $P(\bb)$.  
    Hence, $\card{\I} \ge n - k$.

    For each $I \subseteq [n]$, let $I_{+} \deftobe \setof{n- \card{I} + 1, n - \card{I} + 2, \dotsc, n}$, so that $I_{+}$ comprises the last $\card{I}$ elements of $[n]$.  
    (We will see below that in fact $I = I_{+}$ for all $I \in \I$.)  
    Since
    \begin{equation*}
        \sum_{i = n - \card{I} + 1}^{n} \, S_i
        =
        \sum_{i \in I}
        w_i
        \le
        \sum_{i = n - \card{I} + 1}^{n} w_{i} 
        \le
        \sum_{i = n - \card{I} + 1}^{n} \, S_i
        \qquad
        \text{for all $I \in \I$}
    \end{equation*}
    (where the middle inequality holds because $\bw$ is weakly increasing, and the last inequality holds because $\bw \in P(\bb)$), we have that
    \begin{equation}
    \label{eq: sum over I-plus equality equals sum over I}%
        \sum_{i \in I} w_i 
        =%
        \sum_{i = n - \card{I} + 1}^{n} w_{i}
        \qquad
        \text{for all $I \in \I$}
    \end{equation}
    and, in particular, by \cref{eq: w is a vertex}, that $I_{+} \in \I$ for all $I \in \I$.  
    Moreover, note that \cref{eq: sum over I-plus equality equals sum over I} is an equation of two sums, each of which is a sum of the same number $\card{I}$ of terms, and that the $\ell$\textsuperscript{th} term on the left-hand side of \cref{eq: sum over I-plus equality equals sum over I} is bounded above by the $\ell$\textsuperscript{th} term on the right-hand side for $1 \le \ell \le \card{I}$.
    It follows that the $\ell$\textsuperscript{th} term on the left-hand side equals the $\ell$\textsuperscript{th} term on the right-hand side for $1 \le \ell \le \card{I}$. 
    In particular, the first terms on each side are equal:
    \begin{equation}
    \label{eq: wsubminI equals wsubnminuscardIplus1}%
        w_{\min(I)} = w_{n - \card{I} + 1}
        \qquad
        \text{for all $I \in \I$}.
    \end{equation}
    
    Since $\bw \in P(\bb)$, we have in particular that
    \begin{equation}\label[ineq]{ineq: vertex of Pb}
        \sum_{i= n - \card{I} + 2}^{n} w_{i}
        \le
        \sum_{i = n - \card{I} + 2}^{n} \, S_i 
        \qquad \text{for all $I \subseteq [n]$},
    \end{equation}
    while, from \cref{eq: w is a vertex,eq: sum over I-plus equality equals sum over I},
    \begin{equation}\label{eq: vertex of Pb}
        \sum_{i= n - \card{I} + 1}^{n} w_{i}
        =
        \sum_{i = n - \card{I} + 1}^{n} \, S_i
        \qquad
        \text{for all $I \in \I$}.
    \end{equation}
    Subtracting \cref{ineq: vertex of Pb} from \cref{eq: vertex of Pb} yields that
    \begin{equation}
    \label[ineq]{ineq: lower bound on wsubnminuscardIplus1} %
        w_{n-\card{I} + 1} 
        \ge
        S_{n- \card{I} + 1}
        \qquad \text{for all $I \in \I$}.
    \end{equation}
    Moreover, since $\bw \in P(\bb)$,
    \begin{equation}\label[ineq]{ineq2: vertex of Pb}
        \sum_{i= n - \card{I}}^{n} w_{i}
        \le
        \sum_{i = n - \card{I}}^{n} \, S_i
        \qquad 
        \text{for all $I \subseteq [n]$ such that $\card{I} \le n-1$}.
    \end{equation}
    Subtracting \cref{eq: vertex of Pb} from \cref{ineq2: vertex of Pb} yields
    \begin{equation}
    \label[ineq]{ineq: upper bound on wsubnminuscardI} %
        w_{n-\card{I}} 
        \le
        S_{n-\card{I}}
        \qquad \text{for all $I \in \I$ such that $\card{I} \le n-1$}.
    \end{equation}
    Since $S_{n - \card{I}} < S_{n - \card{I} + 1}$, it follows from \cref{ineq: lower bound on wsubnminuscardIplus1,ineq: upper bound on wsubnminuscardI} that
    \begin{equation}
    \label[ineq]{ineq: wsubnminuscardI less than wsubnminuscardIplus1}
        w_{n-\card{I}} < w_{n-\card{I} + 1}
        \qquad %
        \text{for all $I\in \I$ such that $\card{I} \le n-1$}.
    \end{equation}
   
    We now get that, in fact, $I = I_{+}$ for all $I \in \I$, since otherwise we would have an $I \in \I$ with $\card{I} \le n - 1$ and $\min(I) < \min(I_{+}) = n - \card{I} + 1 $, and so, from \cref{eq: wsubminI equals wsubnminuscardIplus1} and $w_{\min(I)} \le w_{n - \card{I}} \le w_{n - \card{I} + 1}$, we would get that $w_{n - \card{I}} = w_{n - \card{I} + 1}$, contradicting \cref{ineq: wsubnminuscardI less than wsubnminuscardIplus1}.
    Thus, every element of $\I$ is of the form $I = I_{j} \deftobe \setof{j, j+1, \dotsc, n}$ for some $1 \le j \le n$.

    We will now show that $\bw = \y_{k} = (1, \dotsc, 1, S_{k+1}, S_{k+2}, \dotsc, S_{n})$.  
    However, we must separately consider two cases.
    For the first case, suppose that $w_{j} > 1$ for all $j$ such that $I_{j} \in \I$.  
    Since ${w_{1} = \dotsb = w_{k} = 1}$, it follows that $j \ge k+1$ for all $j$ such that $I_{j} \in \I$.  
    In other words, every $I \in \I$ is of the form $I_{j}$ with $k+1 \le j \le n$.
    But there are only $n-k$ such subsets of~$[n]$, and $\card{\I} \ge n - k$, so we must have that $\I$ is \emph{precisely} the set $\setof{I_{k+1}, I_{k+2}, \dotsc, I_{n}}$.  
    Subtracting the \eqref{x-parking}~\emph{equation} for~$I_{j+1}$ from the equation for~$I_{j}$, we get that $w_{j} = S_{j}$ for $k+1 \le j \le n - 1$, and the $I_{n}$ equation itself states that $w_{n} = S_{n}$.  
    Thus, $\bw = \y_{k} \in \Xn$, and moreover $\bw$ satisfies precisely $k + \card{\I} = n$ of the inequalities in \cref{subthm: facet-defining inequalities} with equality.

    This leaves the case in which $w_{j} = 1$ for some $j$ such that $I_{j} \in \I$.  
    From \cref{ineq: lower bound on wsubnminuscardIplus1}, it follows that $S_{j} = 1$.  
    Since $1 \le S_{1} < \dotsb < S_{n}$, we get that $j = 1$ and hence $b_{1} = S_{1} = 1 = w_1$ and $[n] = I_{1} \in \I$.  
    (Recall that we are also now assuming that $n \ge 2$ because we have excluded the degenerate case $n = 1 = b_{1}$ in the remarks following \cref{prop: Xn is n-dimensional}.)  
    Furthermore, $k = 1$ because, since $[n] = I_1 \in \I$, we may subtract the \eqref{x-parking} inequality for $I_{k}$ from the \eqref{x-parking} equation for~$[n]$ to get that
    \begin{equation*}
       k
       =
       \sum_{i=1}^{k} w_{i}
       \ge
       \sum_{i=1}^{k} S_{i}
       \ge
       \sum_{i=1}^{k} i,
    \end{equation*}
    which, together with $k \ge 1$ (because $w_{1} = 1$), implies that $k = 1$.  
    Thus, $\card{\I} \ge n-1$.  
    Now, since $b_{1} = 1$, the set $\I$ does not contain the subset $I_{2}$ because $\card{I_2} = n-1$, so $\I$ must contain all $n-2$ subsets $I_{3}, \dotsc, I_{n}$.%
    \footnote{%
        The claim that there are exactly $n-2$ such subsets is the point at which this argument fails in the degenerate case $n
        = 1 = b_{1}$.
    } %
    Hence, $\I = \setof{[n], I_{3}, \dotsc, I_{n}}$.  
    Subtracting the \eqref{x-parking}~equation for $I_{3}$ from the \eqref{x-parking}~equation for~$[n]$ and applying $w_{1} = 1 = b_{1}$ yields $w_{2} = S_{2}$.  
    For $3 \le j \le n-1$, subtract the \eqref{x-parking}~equation for~$I_{j+1}$ from the equation for~$I_{j}$ to get that $w_{j} = S_{j}$.  
    Finally, the $I_{n}$ equation itself states that $w_{n} = S_{n}$.  
    Thus, we have again found that $\bw = \y_{k} \in \Xn$ and that $\bw$ satisfies precisely $k + \card{\I} = 1 + 1 + (n-2) = n$ of the inequalities in \cref{subthm: facet-defining inequalities} with equality.
      
    Since, in both cases, $\bw \in \Xn$, we now have that $P(\bb) = \Xn$.  
    We have also found that every vertex of $P(\bb)$ equals $\pi(\y_{k})$ for some $\pi \in \SSS$ and $0 \le k \le n$, so we have now proved that every vertex of $\Xn$ has this same form.  
    In addition, we have found that the number of inequalities in \cref{subthm: facet-defining inequalities} that the vertex $\bw$ satisfies with equality is precisely $n$, so $\Xn$ is simple (proving \cref{subthm: Xn is simple}) and moreover no inequality given in \cref{subthm: facet-defining inequalities} that supports $\Xn$ is redundant.
    In particular, since every point of the form $\y_{k}$ with $0 \le k \le n$ satisfies precisely~$n$ of the inequalities in \cref{subthm: facet-defining inequalities} with equality, it follows that every such point is a vertex of $\Xn$.  
    This completes the proof of \cref{subthm: vertices of Xn}.
   
    The final claim is that this is a minimal system of inequalities. 
    Since $\Xn$ is simple, it remains only to show that every inequality in this system supports $\Xn$.  
    Of course, the inequalities of the form~\mbox{$x_{i} \ge 1$} are all satisfied with equality by the point $(1, \dotsc, 1) \in \Xn$.  
    For the \eqref{x-parking}~inequalities, let $I \subseteq [n]$ be given, and let $k \deftobe n - \card{I}$.  
    Then $\y_{k}$ satisfies the \eqref{x-parking}~inequality for $I_{+}$ with equality.  
    Let $\pi \in \SSS$ be any permutation under which the image of $I_{+}$ is $I$.  
    Then $\pi(\y_{k})$ is a point in $\Xn$ that satisfies the \eqref{x-parking}~inequality for~$I$ itself with equality.
\end{proof}

As a corollary of \cref{subthm: vertices of Xn}, we enumerate the vertices of $\Xn$.  
When $b_1=1$ (and so $\y_{0} = \y_{1}$), the number of vertices is the same as in the case of the classical parking-function polytope~$\PF_n$.
However, when $b_1 \ge 2$, an additional $n!$ vertices come into play.

\begin{corollary}\label{cor:number_vertices}
   If $b_{1} = 1$, then the number of vertices of $\Xn$ is $n!\parens{\frac{1}{1!}+\dotsb+\frac{1}{n!}}$. 
   Otherwise, if $b_{1} \ge 2$, then the number of vertices of $\Xn$ is $n!\parens{\frac{1}{0!}+\frac{1}{1!}+\dotsb+\frac{1}{n!}}$.
\end{corollary}

\begin{proof}
    When $b_{1} \ge 2$, or when $b_{1} = 1$ and $k \ge 2$, each vertex~$\bv$ of $\Xn$ of the form $\pi(\y_{k})$ ($\pi \in \SSS$) corresponds to the unique pair $(I, \sigma)$ of the $(n-k)$-subset $I \subseteq [n]$ of entries of~$\bv$ that do not equal $1$ and the bijection $\sigma \maps [n] \setminus [k] \to I$ that gives the position $\pi(j)$ of the value $S_{j}$ in $\bv$ for $k+1 \le j \le n$.  
    On the other hand, when $b_{1} = 1$ and $k \in \setof{0, 1}$, so that $\y_{0} = \y_{1}$, each vertex of the form $\pi(\y_{k})$ corresponds to the unique permutation in $\SSS$ (namely, $\pi$) that gives the position $\pi(j)$ of the value~$S_{j}$ in $\bv$ for $1 \le j \le n$.  
    Thus, when $b_{1} = 1$, the number of vertices is $n!  + \sum_{k = 2}^{n} \binom{n}{n-k} (n-k)!$, while, when $b_{1} \ge 2$, the number of vertices is $\smash{\sum_{k=0}^{n} \binom{n}{n-k} (n-k)!}$.
\end{proof}

From the characterization of the vertices and facets of $\Xn$, it is straightforward to describe the facets that contain a given nondecreasing vertex $\y_{k}$.  
To this end, define the following notation:
\begin{alignat*}{2}
    L_{j} 
    &\deftobe%
    \bigbraces{(x_{1}, \dotsc, x_{n}) \in \Xn \sst x_{j} = 1} &&
    \qquad \text{for $1 \le j \le n$,} \\
    U_{j}
    &\deftobe%
    \bigbraces{(x_{1}, \dotsc, x_{n}) \in \Xn \sst \textstyle\sum_{i =
    j}^{n} x_{i} = \sum_{i = j}^{n} S_{i}} && \qquad \text{for $1 \le
    j \le n$, with $j \ne 2$ if $b_{1} = 1$.}
\end{alignat*}
Thus, $L_{j}$ is the facet of $\Xn$ defined by the inequality $x_{j} \ge 1$, and $U_{j}$ is the facet of $\Xn$ defined by the \eqref{x-parking}~inequality for the subset $I = I_{j} \deftobe \setof{j, j+1, \dotsc, n}$ (where $j \ne 2$ if $b_{1} = 1$).  
In particular, it is routine to check that the facets that contain the vertex $\y_{k}$ are as follows.

\begin{lemma}
    \label{lem:facets containing a vertex}
    \mbox{}
    \begin{enumerate}[label=(\alph*), ref={\thetheorem(\alph*)}]
        \item  
        \label[sublem]{sublem: facets b1 ge 2 or b1 is 1 and k ge 2}%
        If $b_{1} \ge 2$, or if $b_{1} = 1$ and $k \ge 2$, then the $n$ distinct facets that contain $\y_{k}$ are
        \begin{equation*}
            L_{1}, L_{2}, \dotsc, L_{k}, U_{k+1}, U_{k+2}, \dotsc, U_{n}.
        \end{equation*}
        
        \item  
        \label[sublem]{sublem: facets b1 is 1 and k is 1}%
        If $b_{1} = 1$ and $k \in \setof{0, 1}$ (\emph{i.e., $\y_k = \y_0 = \y_1$}), then the $n$ distinct facets that contain $\y_{k}$ are
        \begin{equation*}
            L_{1}, U_{1}, U_{3}, U_{4}, \dotsc, U_{n}.
        \end{equation*}
    \end{enumerate}
\end{lemma}

\begin{proof}
    Since $\Xn$ is simple and $n$-dimensional, it suffices to confirm that the $n$ given facets contain~$\y_{k}$.
\end{proof}

We will now use \cref{lem:facets containing a vertex} to describe the precise conditions under which two vertices of~$\Xn$ share an edge.
An informal summary of the conditions is as follows: Two vertices $\bv$ and $\bw$ share an edge if and only if either
\begin{enumerate}
    \item 
    the vertices are the same, except that the minimum entry $>1$ in one vertex is a $1$ in the other vertex; or

    \item 
    the vertices are the same, except that two entries of the form $S_{j-1}$ and $S_j$ swap positions.
\end{enumerate}
In addition, we find that the edge vectors are all parallel to roots in the root~system~$B_n$, with lengths given directly by the entries of the vector $\bb$ defining $\Xn$ or by their partial sums.

\begin{proposition}\label{prop:edges}
    Let $\bv$ and $\bw$ be vertices of $\Xn$.
    Fix $\pi \in \SSS$ and $k \in \Z$ with $0 \le k \le n$ such that $\bv = \pi(\y_{k})$.
    \begin{enumerate}[label=(\alph*), ref={\thetheorem(\alph*)}]
        \item
        \label[subprop]{subprop: edges b1 ge 2 or b1 is 1 and k ge 2}%
        If $b_{1} \ge 2$, or if $b_{1} = 1$ and $k \ge 2$, then $\bv$ and $\bw$ share an edge of $\Xn$ if and only if either
        \begin{enumerate}[label=(\arabic*)]
            \item  
            $\bw = (\pi \circ (j,k))(\y_{k-1})$ for some $1 \le j \le k$, in which case \[\bw - \bv = (S_{k} - 1) \ee_{\pi(j)}\] (informally, the vertices are the same, except that one of the $1$ entries in $\bv$ is $S_{k}$ in $\bw$); or
            
            \item  
            $\bw = (\pi \circ (j-1,j))(\y_{k})$ for some $k + 2 \le j \le n$, in which case \[\bw - \bv = b_{j}(\ee_{\pi(j-1)} - \ee_{\pi(j)})\] (informally, the vertices are the same, except that entries $S_{j-1}$ and $S_{j}$ swap positions);%
            \footnote{%
                When $b_{1} = 1$ and $k \in \setof{0,1}$ (so that
                $\y_{k} = \y_{0} = \y_{1}$), we must use $k = 0$ for
                this equation to hold.
            } %
            or

            \item  
            $k \le n-1$ and $\bw = \pi(\y_{k+1})$, in which case \[\bw - \bv = - (S_{k+1} - 1) \ee_{\pi(k+1)}\] (informally, the vertices are the same, except that the $S_{k+1}$ entry in $\bv$ is a $1$ in $\bw$).%
            \footnote{%
                When $b_{1} = 1$ and $k \in \setof{0,1}$ (so that
                $\y_{k} = \y_{0} = \y_{1}$), we must use $k = 1$ for
                this equation to hold.
            } %
        \end{enumerate}

        \item  
        \label[subprop]{subprop: edges b1 is 1 and k is 1}%
        If $b_{1} = 1$ and $k \in \setof{0, 1}$ (\emph{i.e., $\y_k = \y_0 = \y_1$}), then $\bv$ and $\bw$ share an edge of $\Xn$ if and only if either
        \begin{enumerate}[label=(\arabic*)]
            \item  
            $\bw = (\pi \circ (j-1,j))(\y_{1})$ for some $2 \le j \le n$, in which case \[\bw - \bv = b_{j}(\ee_{\pi(j-1)} - \ee_{\pi(j)})\] (informally, the vertices are the same, except that entries $S_{j-1}$ and $S_{j}$ swap
            positions); or
            
            \item  
            $\bw = \pi(\y_{2})$,
            in which case \[\bw - \bv = -b_{2} \ee_{\pi(2)}\] (informally, the vertices are the same, except that the $S_{2}$ entry in $\bv$ is a $1$ in $\bw$).
        \end{enumerate}
    \end{enumerate}
\end{proposition}

\begin{proof}
    For both parts of the \lcnamecref{prop:edges} to be proved, note that we will be done once we have proved the ``if'' direction, since we will then have found $n$ vertices that are adjacent to $\bv$, which must be all of the vertices that are adjacent to $\bv$, because $\Xn$ is simple and of dimension $n$ by \cref{subthm: Xn is simple,prop: Xn is n-dimensional}.  
    Thus, the ``only if'' direction will follow immediately.  
    Note also that the calculation of the difference vector $\bw - \bv$ under each condition is straightforward.

    To prove the ``if'' direction of \cref{subprop: edges b1 ge 2 or b1 is 1 and k ge 2}, suppose that $b_{1} \ge 2$, or $b_{1} = 1$ and $k \ge 2$, as in \cref{sublem: facets b1 ge 2 or b1 is 1 and k ge 2}.  
    Thus, the facets that contain $\y_{k}$ are
    \begin{equation}
        L_{1}, L_{2}, \dotsc, L_{k}, U_{k+1}, U_{k+2}, \dotsc, U_{n}.
        \label{list: facets containing a vertex}
    \end{equation}
    By the $\SSS$ symmetry of $\Xn$, the vertices $\bv$ and $\bw$ share an edge if and only if $\pi^{-1}(\bw)$ lies on all but one of these facets.  
    Now, if $\bw = (\pi \circ (j, k))(\y_{k-1})$ for some $1 \le j \le k$, then $\pi^{-1}(\bw) = (j, k)(\y_{k-1})$ lies on all facets in~\eqref{list: facets containing a vertex} except~$L_{j}$.  
    If $\bw = (\pi \circ (j-1,j))(\y_{k})$ for some $k + 2 \le j \le n$, then~$\pi^{-1}(\bw) = (j-1, j)(\y_{k})$ lies on all facets in~\eqref{list: facets containing a vertex} except~$U_{j}$.  
    Finally, if $k \le n-1$ and $\bw = \pi(\y_{k+1})$, then $\pi^{-1}(\bw) = \y_{k+1}$ lies on all facets in~\eqref{list: facets containing a vertex} except~$U_{k+1}$.

    To prove the ``if'' direction of \cref{subprop: edges b1 is 1 and k is 1}, suppose that $b_{1} = 1$ and $k \in \setof{0, 1}$, as in \cref{sublem: facets b1 is 1 and k is 1}.  
    Thus, the facets that contain $\y_{k} = \y_0 = \y_1$ are
    \begin{equation}
        L_{1}, U_{1}, U_{3}, U_{4}, \dotsc, U_{n}.
        \label{list: facets containing a vertex2}
    \end{equation}
    If $\bw = (\pi \circ (1,2))(\y_{0})$, then~$\pi^{-1}(\bw) = (1, 2)(\y_{0})$ lies on all facets in~\eqref{list: facets containing a vertex2} except~$L_{1}$.  
    If $\bw = (\pi \circ (j-1,j))(\y_{0})$ for some $3 \le j \le n$, then~$\pi^{-1}(\bw) = (j-1, j)(\y_{0})$ lies on all facets in~\eqref{list: facets containing a vertex2} except~$U_{j}$.  
    Finally, if $\bw = \pi(\y_{2})$, then $\pi^{-1}(\bw) = \y_{2}$ lies on all facets in~\eqref{list: facets containing a vertex2} except~$U_{1}$.
\end{proof}

Recall that, given a vertex $\bv$ of a polytope $P$, the \defing{tangent cone} $\Tan{\bv}(P)$ at $\bv$ is the cone generated by the edge vectors pointing from $\bv$ to its neighbors in the edge graph of $P$.  
That is,
\begin{equation*}
    \Tan{\bv}(P) 
    \deftobe 
    \Bigg\lbrace
        \sum_{%
            \substack{%
                \bw \in \Ver(P) \sst \\
                \bw \adjto \bv
            }%
        }%
        \lambda_{\bw}(\bw - \bv) 
        \sst
        \text{%
            $\lambda_{\bw} \ge 0$ for all vertices
            $\bw \adjto \bv$
        }%
    \Bigg\rbrace,%
\end{equation*}
where we write $\bw \adjto \bv$ when $\bw$ and $\bv$ are adjacent vertices of $P$. 
As a corollary of \cref{prop:edges}, we get an explicit description of the tangent cone at each vertex of $\Xn$.

\begin{corollary}
\label{cor: generators of tangent cone of Xn}
    Let $\bv$ be a vertex of $\Xn$.  
    Fix $\pi \in \SSS$ and $k \in \Z$ with $0 \le k \le n$ such that $\bv = \pi(\y_{k})$.  
    Let $\Tan{\bv} \deftobe \Tan{\bv}(\Xn)$ be the tangent cone at $\bv$.
    \begin{enumerate}[label=(\alph*), ref={\thetheorem(\alph*)}]
        \item  
        Suppose that $b_{1} \ge 2$, or that $b_{1} = 1$ and $k \ge 2$.
        Then $\Tan{\bv}$ is generated by
        \begin{equation*}
            \setof{\ee_{\pi(j)} \sst 1 \le j \le k} 
            \quad
            \cup 
            \quad
            \setof{-\ee_{\pi(k+1)}}
            \quad
            \cup 
            \quad
            \setof{\ee_{\pi(j-1)} - \ee_{\pi(j)} \sst k+2 \le j \le n}
        \end{equation*}
        if $k \le n-1$, and by
        $
            \setof{\ee_{j} \sst 1 \le j \le n} 
        $
        if $k = n$.

        \item  
        Suppose that $b_{1} = 1$ and $k \in \setof{0, 1}$.  
        Then $\Tan{\bv}$ is generated by
        \begin{equation*}
            \setof{\ee_{\pi(j-1)} - \ee_{\pi(j)} \sst 2 \le j \le n}
            \quad
            \cup
            \quad
            \setof{-\ee_{\pi(2)}}.
        \end{equation*}

    \end{enumerate}
\end{corollary}

\subsection{$\Xn$ is a generalized permutahedron}

Generalized permutahedra (sometimes spelled \emph{permutohedra}) form a widely studied class of polytopes with a rich theory and many interesting subclasses.  
One definition of a generalized permutahedron is that it is a polyhedron whose normal fan is a coarsening of the braid arrangement.
For sources on generalized permutahedra, we recommend \cite{doker2011geometry, JockemkoRavichandran, Pos, PRW}.
We use an equivalent characterization.

\begin{definition} 
A polytope is a \emph{generalized permutahedron} if and only if all edges are parallel to $\ee_i-\ee_j$ for some standard basis vectors $\ee_i$ and $\ee_j$.
\end{definition}

We show that the lifting of $\Xn$ is a generalized permutahedron, from which it, and $\Xn$ itself, immediately inherit a variety of combinatorial and geometric results and properties.  
Let $B \deftobe \sum_{k=1}^n S_{k}$, and let $\ell \maps \R^{n} \to \R^{n+1}$ be the affine-linear map that ``vertically projects'' $\R^{n}$ onto the hyperplane $H=\{\mathbf{x} \in \R^{n+1}: x_{1} + \dotsb + x_{n} + x_{n+1} = B\}$ via $(x_{1}, \dotsc, x_{n}) \mapsto (x_{1}, \dotsc, x_{n}, B - \sum_{i=1}^{n} x_{i})$. 
Using this mapping, Proposition \ref{prop:lifted_vertex_set_of_Xn} below follows.

\begin{definition}
Let $\Xnb \deftobe \{\ell(\mathbf{x}) : \mathbf{x}\in\Xn\}$ be the lifting of $\Xn$ to the hyperplane $H$. 
We call $\Xnb$ the \defing{lifted $\bb$-parking-function polytope}.   
\end{definition}

\begin{proposition}\label{prop:lifted_vertex_set_of_Xn}
    The vertex set of $\Xnb$ is $\{\ell(\bv) \sst \text{$\bv$ is a vertex of $\Xn$}\}$.  
    Furthermore, the minimal inequality description of $\Xnb$ is the minimal inequality description of $\Xn$ together with the additional equality $x_1 + \dotsb + x_n + x_{n+1} = B$.
\end{proposition}

Building off this combinatorial structure, we have the main result of this subsection.

\begin{theorem} \label{thm: x-park_is_GPerm}
The lifted $\bb$-parking-function polytope $\Xnb$ is a generalized permutahedron.
\end{theorem}

\begin{proof}
Proposition~\ref{prop:lifted_vertex_set_of_Xn} implies that the edges of $\Xn$ lift to edges of $\Xnb$ in the natural way: $\{\mathbf{x},\mathbf{y}\}$ is an edge of $\Xn$ if and only if $\{\ell(\mathbf{x}),\ell(\mathbf{y})\}$ is an edge of $\Xnb$.
If $\{\mathbf{x},\mathbf{y}\}$ is an edge of $\Xn$ with $\mathbf{x}-\mathbf{y} = c(\ee_i-\ee_j)$, then $ \sum_{i=1}^n x_i = \sum_{i=1}^n y_i$. Thus $\ell(\mathbf{x})_{n+1} = B - \sum_{i=1}^n x_i = B - \sum_{i=1}^n y_i$. so $\ell(\mathbf{x})-\ell(\mathbf{y})=c(\ee_i-\ee_j)$.
If $\{\mathbf{x},\mathbf{y}\}$ is an edge of $\Xn$ with $\mathbf{x}-\mathbf{y} = c\ee_j$, then $ \sum_{i=1}^n x_i - \sum_{i=1}^n y_i = c$. Thus, $\ell(\mathbf{x})_{n+1} = B - \sum_{i=1}^n x_i = B - \sum_{i=1}^n y_i - c = \ell(\mathbf{y})_{n+1}-c$, so $\ell(\mathbf{x})-\ell(\mathbf{y}) = c(\ee_j-\ee_{n+1})$. By Proposition~\ref{prop:edges}, all edges of $\Xn$ are of this form.
Hence, $\Xnb$ is a generalized permutahedron.
\end{proof}

We may now apply established results about generalized permutahedra to $\Xnb$. 
One such result is \cite[Corollary 4.8]{JockemkoRavichandran}, which provides a formula for the number of lattice points of a generalized permutahedron.
Since lifting a polytope does not change the number of lattice points of a polytope, by applying the formula we obtain the number of lattice points of $\Xn$. 
Furthermore, we note that the enumeration of $\bb$-parking functions is not completely known \cite{Yan2,Yan}, but by applying the formula of \cite[Corollary 4.8]{JockemkoRavichandran}, we obtain an upper-bound for the number of $\bb$-parking functions since all $\bb$-parking functions will be lattice points in $\Xn$.

Additionally, by a theorem of Backman and Liu \cite{backman2023regular}, \cref{thm: x-park_is_GPerm} implies that $\Xnb$, and hence $\Xn$ itself, admits a regular unimodular triangulation.
For a definition of and further details about regular unimodular triangulations, see \cite{HaasePaffenholzPiechnikSantos}.
We can also express $\Xnb$ as a Minkowski sum of simplices.  
In order to do that, we begin by expressing $\Xnb$ in the notation introduced by Postnikov~\cite{Pos} for generalized permutahedra. 
Following from \cite[Section 6]{Pos}, every generalized permutahedron can be expressed as
\[
    P_n^{\cZ}(\{z_I\})\deftobe
    \left\{(x_1,\ldots,x_n)\in\R^n \sst
    \sum_{i=1}^n x_i=z_{[n]},~\sum_{i\in I}x_i\ge z_{I}~\text{ for any subset }I\subset [n]\right\},
\]
for some collection of real parameters $\{z_I\}$ over subsets $I$ of $[n]$ with $z_{\emptyset}=0$.

\begin{proposition} \label{prop: Z-expression}
    The lifted $\bb$-parking-function polytope $\Xnb$ is the generalized permutahedron $P_{n+1}^{\cZ}(\setof{z_I})$ in which the values of the parameters $z_{I}$ for $\emptyset \ne I \subseteq [n+1]$ are given by
    \begin{equation*}
        z_{I}
        \deftobe
        \begin{dcases*}
            \card{I}
            & if $n+1 \notin I$, \\
            S_{1} + S_{2} + \dotsb + S_{\card{I}-1}
            & if $n+1 \in I$.
        \end{dcases*}
    \end{equation*}
\end{proposition}

\begin{proof}
    From \cref{prop:lifted_vertex_set_of_Xn}, consider the defining inequalities, 
\[
    \sum_{i\in I}x_i\le S_n+S_{n-1}+\dotsb+S_{n-|I|+1},
\]
for any $I$ with  $|I|=k$, $n+1\notin I$, and $x_1+\ldots+x_{n+1}=\sum_{i=1}^n S_i$. 
Subtracting these two yields the inequality
\[
    \sum_{i\in [n+1]\setminus I}x_i\ge S_1+S_2+\dotsb+S_{|[n+1]\setminus I|-1}.
\]
This shows that $\Xnb$ can be described in terms of the following inequalities:
\[
    \begin{cases}
        \displaystyle x_i\ge 1, &  \text{ for }1\le i \le n,\\
        \displaystyle\sum_{i\in I}x_i\ge \sum_{i=1}^{|I|-1}S_i, &  \text{ for all $I\subset[n+1]$ with }n+1\in I,\\
        \displaystyle\sum_{i=1}^{n+1} x_i=\sum_{i=1}^n S_i, & 
    \end{cases}
\]
since they recover the defining inequalities in \cref{thm:inequality_description}. 
To match with Postnikov's definition of $P_{n+1}^{\cZ}(\{z_I\})$, we can add the following redundant inequalities obtained from $x_i\ge 1$ which have no effect on the solutions of the above system of inequalities:
\[
    \sum_{i\in I}x_i\ge |I| \quad \text{ for all }I\subset[n+1]\text{ s.t. }n+1\notin I.
\]
These inequalities give all the values of $z_I$ for all nonempty $I\subset[n+1]$.
\end{proof}

In \cite{ardila2011matroid}, Ardila, Benedetti, and Doker prove that any generalized permutahedron can be written as a signed Minkowski sum of scaled standard simplices. 
Specifically, \[P_n^{\cZ}(\{z_I\}) = P^{\Y}_n(\{y_I\}) \deftobe \sum_{I \subseteq [n]} y_I \, \bigtriangleup_I,\] where $y_I = \sum_{J \subseteq I} (-1)^{\card{I}-\card{J}} z_J$ and $\bigtriangleup_I \deftobe \conv\{\ee_i : i \in I\}$,  for each $I \subseteq [n]$. 
In the case that $y_I \geq 0$ for all $I$, we call this a \defing{$\mathcal{Y}$-generalized permutahedron}, which enjoys a variety of special properties including Ehrhart positivity and $h^*$-real-rootedness \cite{Pos}; for more on these properties, see \cite{BrandenSolus, Braun, Liu}.

\begin{proposition} \label{prop: y-expression}
    \begin{sloppypar}
    The lifted $\bb$-parking-function polytope $\Xnb$ is the signed Minkowski sum ${\sum_{I\subseteq [n+1]}y_I \bigtriangleup_I}$, where
    \begin{equation*}
        y_I
        \deftobe
        \begin{dcases*}
            1  & if $n+1 \notin I$ and $\card{I} = 1$, \\
            0  & if $I = \setof{n+1}$, or if $n+1 \notin I$ and $\card{I} \ge 2$, \\
            b_{1}-1 & if $n+1 \in I$ and $\card{I}=2$, \\
            \textstyle\sum_{j = 0}^{\card{I}-3}
            (-1)^{\card{I}+j-1}
            \begin{psmallmatrix}
               \! \card{I}-3 \! \\
               \! j \!
            \end{psmallmatrix}
            b_{j+2}
            & if $n+1 \in I$, $\card{I} \ge 3$.
        \end{dcases*}
    \end{equation*}
    \end{sloppypar}
\end{proposition}

\begin{proof}
    We write $\Xnb=P_{n+1}^{\cZ}(\{z_I\})$, with $z_I$ as given in \cref{prop: Z-expression}.  
    By \cite[Prop.~2.2.4]{doker2011geometry}, this becomes $P_{n+1}^{\cZ}(\{z_I\})=\sum_{I\subseteq[n+1]}y_I\bigtriangleup_I$, where the $y_I$s are determined by $z_I=\sum_{J\subseteq I}y_J$  for all nonempty $I\subseteq[n+1]$. 

First note that $y_{\{i\}}=z_{\{i\}}=
    \begin{cases*}
        1 & if $1\le i\le n$,\\
        0 & if $i=n+1$.
    \end{cases*}$
    
\noindent For $I\subseteq[n+1]$ not containing $n+1$ with $\card{I} \ge 2$, we apply the principle of inclusion--exclusion to get
    \begin{equation*}
    y_I = \sum_{J \subseteq I}
        (-1)^{\card{I} - \card{J}} z_J
        = \sum_{k = 0}^{\card{I}} (-1)^{\card{I} - k} \binom{\card{I}}{k} k
        = 0.
    \end{equation*}
    For $I\subseteq[n+1]$ with $n+1 \in I$, $\card{I}\ge 2$, then
    again by the principle of inclusion--exclusion,
    \begin{align}
        y_I
        &=
        \sum_{J \subseteq I}
        (-1)^{\card{I} - \card{J}} z_J
        =
        \sum_{J \subseteq I \setminus \setof{n+1}}
        (-1)^{\card{I} - \card{J}} z_J
        +
        \sum_{\substack{%
            J \subset I \\
            n+1\in J
        }}
        (-1)^{\card{I}-\card{J}} z_J \notag \\
        &=
        -y_{I \setminus \setof{n+1}}
        +
        \sum_{k = 1}^{\card{I} - 1}
        (-1)^{\card{I} - k - 1} \binom{\card{I} - 1}{k}
        \sum_{\ell = 1}^k S_{\ell}. 
        \label{eq: 1}
    \end{align}
    When $\card{I}=2$, $y_I = -1 + b_1$.  
    For $\card{I}\ge 3$, the second term on the right-hand side of \cref{eq: 1} can be rewritten as
    \begin{align*}
        &\sum_{\ell=1}^{\card{I}-1}S_{\ell}\left(\sum_{k=\ell}^{\card{I}-1}(-1)^{\card{I}-1-k}\binom{\card{I}-1}{k}\right)=\sum_{\ell=1}^{\card{I}-1}S_{\ell}\left(\sum_{k=0}^{\card{I}-1-\ell}(-1)^k\binom{\card{I}-1}{k}\right)\\
        &\quad = \sum_{\ell=1}^{\card{I}-1}\left(\sum_{i=1}^\ell b_i\right)(-1)^{\card{I}-1-\ell}\binom{\card{I}-2}{\card{I}-1-\ell}=\sum_{i=1}^{\card{I}-1}b_i\sum_{\ell=i}^{\card{I}-1}(-1)^{\card{I}-1-\ell}\binom{\card{I}-2}{\card{I}-1-\ell}\\
        &\quad = \sum_{i=1}^{\card{I}-1}b_i\sum_{\ell=0}^{\card{I}-1-i}(-1)^{\ell}\binom{\card{I}-2}{\ell}=\sum_{i=2}^{\card{I}-1}(-1)^{\card{I}-1-i}\binom{\card{I}-3}{\card{I}-1-i}b_i.
    \end{align*}
    Hence, for $\card{I}\ge 3$, 
    \begin{equation*}
        y_I
        =
        \sum_{i=2}^{\card{I} - 1}
        (-1)^{\card{I} - 1 - i}
        \binom{\card{I} - 3}{\card{I} - 1 - i}b_i
        =
        (-1)^{\card{I} - 1}
        \sum_{i = 0}^{\card{I} - 3}
        (-1)^i 
        \binom{\card{I} - 3}{i}
        b_{i + 2}.
        \qedhere
    \end{equation*}
\end{proof}

\Cref{prop: y-expression} allows us to determine from $\bb = (b_1, \dotsc, b_n)$ whether $\Xnb$ is a $\Y$-generalized permutahedron.  
In particular, the liftings of $\PF_n$ and of the specific $\bb$-parking-function polytope $\mathfrak{X}_n(a, b)$ studied in \cite{HanadaLentferVindas} (where $(a,b) \deftobe \bb = (a, b, b, \dotsc, b)$) are both $\Y$-generalized permutahedra.

\subsection{$h$-Polynomial of $\Xn$} \label{subsec:h-poly}

We now consider the $h$-polynomial of the $\bb$-parking-function polytope.
Given an $n$-dimen\-sional polytope~$P$ and $0 \leq i \leq n$, let~$f_i(P)$ denote the number of $i$-dimensional faces of~$P$.  
The $f$-polynomial of~$P$ is then defined as $f(P;t) \deftobe \sum_{i=0}^n f_i(P)t^i$ and the $h$\nobreakdash-polynomial of~$P$ is the result of applying a change of basis given by ${h(P;t) \deftobe f(P;t-1)}$.

The stellohedron $\St_n$ is a simple generalized permutahedron first defined by Postnikov, Reiner, and Williams \cite{PRW} as the nestohedron for a certain building set (see \cref{sec:bulding_set} for more details), which is intimately connected to $\Xn$.
The $h$-polynomial of $\St_n$ is the \textit{binomial Eulerian polynomial} 
\[
    \widetilde{A}_n(z)=1+z\sum_{k=1}^n\binom{n}{k}A_k(z),
\]
where $A_k(z)=\sum_{\sigma\in\mathfrak{S}_k}z^{\des(\sigma)}$ is the Eulerian polynomial \cite{PRW}.

From \cref{subthm: Xn is simple,thm: x-park_is_GPerm}, we have that $\Xnb$ is a simple generalized permutahedron.  
By \cite[Theorem 4.2]{PRW}, the $h$-polynomial of a simple generalized permutahedron can be computed using its so-called \emph{vertex posets}, which are defined for every vertex of the polytope as follows.

For a generalized permutahedron $P$ of dimension $n$ in $\mathbb{R}^{n+1}$, the normal cone of every vertex~$v$ of $P$ is the intersection of the half spaces in $\R^{n+1}/(1, \dotsc, 1)\R$ defined by a set of inequalities of the form $x_i \le x_j$, where $(i, j) \in [n+1]^{2}$.  
One can associate to each vertex $v$ a \emph{vertex poset} $Q_v$ by setting $i\le_{Q_v} j$ whenever $x_i\le x_j$ in the normal cone of~$v$.  
We denote covering relations by $\lessdot$.  
The \emph{descent set} $\Des(Q_v)$ of $Q_v$ is the set of ordered pairs $(i,j)$ such that $i \lessdot_{Q_v} j$ but $i>_{\mathbb{Z}}j$.  
The \emph{number of descents} of $Q_v$ is $\des(Q_v) \deftobe \card{\Des(Q_v)}$.  
When $P$ is a simple generalized permutahedron, by \cite[Theorem 4.2]{PRW}, its $h$-polynomial is given by
\begin{equation*}
    h(P;z)=\sum_{v\in \Ver(P)}z^{\des(Q_v)}.
\end{equation*} 
Since the face numbers of $\Xn$ do not change after the lifting $\ell$, we may use this equation to compute the $h$-polynomial of $\Xn$.
We characterize the vertex poset for every vertex of $\Xnb$.

\begin{figure}[H]
    \newlength{\subfigwidth}
    \setlength{\subfigwidth}{0.3\textwidth}
    \centering
   
    \begin{subfigure}[b]{\subfigwidth} 
        \centering
        \footnotesize
        \begin{tikzpicture}[scale=0.7]
            \node (dot) at (0,0) {$\phantom{(1)}$};
            \node (1) at (0,1) {$n+1$};
            \node (2) at (0,2) {$\pi(1)$};
            \node (3) at (0,3) {$\pi(2)$};
            \node (6) at (0,4) {$\rvdots$};
            \node (4) at (0,5) {$\pi(n-1)$};
            \node (5) at (0,6) {$\pi(n)$};
            
            \draw (1)--(2)--(3);
            \draw (3)--(6);
            \draw (6)--(4);
            \draw (4)--(5);
        \end{tikzpicture}
        \caption{The $k=0$ case}
        \label{fig:poset k=0}
    \end{subfigure}
    \begin{subfigure}[b]{\subfigwidth} 
        \centering
        \footnotesize
        \begin{tikzpicture}[scale=0.7]
            \node (1) at (-0.5,0) {$\pi(1)$};
            \node (2) at (0.5,0) {$\pi(2)$};
            \node (dot) at (1.25,0) {$\ldots$};
            \node (3) at (2,0) {$\pi(k)$};
            \node (4) at (0.75,1) {$n+1$};
            \node (5) at (0.75,2) {$\pi(k+1)$};
            \node (6) at (0.75,3) {$\pi(k+2)$};
            \node (9) at (0.75,4) {$\rvdots$};
            \node (7) at (0.75,5) {$\pi(n-1)$};
            \node (8) at (0.75,6) {$\pi(n)$};
            
            \draw (1)--(4);
            \draw (2)--(4);
            \draw (3)--(4);
            \draw (4)--(5)--(6);
            \draw (6)--(9);
            \draw (9)--(7);
            \draw (7)--(8);
        \end{tikzpicture}
        \caption{The $1\le k\le n-1$ case}
        \label{fig:vertex poset rake permute}
    \end{subfigure}
    \begin{subfigure}[b]{\subfigwidth} 
        \centering
        \footnotesize
        \begin{tikzpicture}[scale=0.7]
            \node (1) at (-0.5,0) {$\phantom{(}1\phantom{)}$};
            \node (2) at (0.5,0) {$2$};
            \node (dot) at (1.25,0) {$\ldots$};
            \node (3) at (2,0) {$n$};
            \node (4) at (0.75,1) {$n+1$};
                        
            \draw (1)--(4);
            \draw (2)--(4);
            \draw (3)--(4);
            
        \end{tikzpicture}
        \caption{The $k=n$ case}
        \label{fig:poset k=n}
    \end{subfigure}
    \caption{The vertex poset $Q_{\overline{\bv}}$ in \cref{sublem: Vertex posets b1
    ge 2}
    }
    \label{fig:vertex poset b1 ge 2}
\end{figure}

\begin{lemma} \label{lem: Vertex posets}
    Let $\overline{\bv}$ be a vertex of $\Xnb$.  
    Fix $\pi \in \SSS$ and $k \in \Z$ with $0 \le k \le n$ such that $\overline{\bv}$ is the image under the lifting $\ell$ of the vertex $\pi(\y_{k})$ of $\Xn$.
    \begin{enumerate}[label=(\alph*), ref={\thetheorem(\alph*)}]
        \item 
        \label[sublem]{sublem: Vertex posets b1 ge 2}%
        If $b_{1} \ge 2$, or if $b_{1} = 1$ and $k \ge 2$, then the vertex poset $Q_{\overline{\bv}}$ is defined by the covering relations
        \begin{alignat*}{2}
            \pi(j)
            &\lessdot
            n+1 && \qquad \text{for $1 \le j \le k$}, \\ 
            n+1
            &\lessdot
            \pi(k+1), \\
            \pi(j-1)
            &\lessdot
            \pi(j) && \qquad \text{for $k+2 \le j \le n$}.
        \end{alignat*}
        (See the Hasse diagrams in
        \cref{fig:vertex poset b1 ge 2}.)

        \item  
        \label[sublem]{sublem: Vertex posets b1 = 1}%
        If $b_{1} = 1$ and $k \in \setof{0,1}$ (in which case $\y_k = \y_{0} = \y_{1}$), then the vertex poset $Q_{\overline{\bv}}$ is defined by the covering relations
        \begin{align*}
            \pi(1) 
            & \lessdot%
            \pi(2)
            \lessdot \dotsb \lessdot
            \pi(n), \\
            n+1
            & \lessdot%
            \pi(2).
        \end{align*}
        (See the Hasse diagram in \cref{fig:vertex poset b_1 = 1}.)
    \end{enumerate}
\end{lemma}

\begin{figure}[H]
    \centering

    \footnotesize
    \begin{tikzpicture}[scale=0.7]
        \node (1) at (0,0) {$\pi(1)$};
        \node (2) at (2,0) {$n+1$};
        \node (3) at (1,1) {$\pi(2)$};
        \node (4) at (1,2) {$\pi(3)$};
        \node (7) at (1,3) {$\rvdots$};
        \node (5) at (1,4) {$\pi(n-1)$};
        \node (6) at (1,5) {$\pi(n)$};
        
        \draw (1)--(3);
        \draw (2)--(3);
        \draw (3)--(4);
        \draw (4)--(7);
        \draw (7)--(5);
        \draw (5)--(6);
    \end{tikzpicture}
    \caption{The vertex poset $Q_{\overline{\bv}}$ in \cref{sublem: Vertex posets b1 = 1}}
    \label{fig:vertex poset b_1 = 1}
\end{figure}

\begin{proof}
    Let $\bv \deftobe \pi(\y_k)$, so that $\overline{\bv} = \ell(\bv)$.  
    The normal cone at $\overline{\bv}$ is the polar of the tangent cone at $\overline{\bv}$.  
    From the formula for the affine projection~$\ell$, we have that $\ell(\bw) - \ell(\bv) = L(\bw - \bv)$ for all $\bw \in \R^{n}$, where $L \maps \R^{n} \to \R^{n+1}$ is the linear map $(x_{1}, \dotsc, x_{n}) \mapsto (x_{1}, \dotsc, x_{n}, -\sum_{i=1}^{n} x_{i})$. 
    To compute the tangent cone of $\Xnb$ at $\overline{\bv}$, we apply~$L$ to the generators of $\Tan{\bv}(\Xn)$ given in \cref{cor: generators of tangent cone of Xn}.  
    Now,
    \begin{align*}
        L(\ee_{\pi(j)}) 
        &=%
        \ee_{\pi(j)} - \ee_{n+1}, \\
        L(-\ee_{\pi(j)}) 
        &=%
        -\ee_{\pi(j)} + \ee_{n+1}, \\
        L(\ee_{\pi(j-1)} - \ee_{\pi(j)})
        &=%
        \ee_{\pi(j-1)} - \ee_{\pi(j)},
    \end{align*}
    for $1 \le j \le n$.\footnote{Note that, in an abuse of notation, we are writing $\ee_{j}$ for the $j$\textsuperscript{th} standard basis vector in $\R^{n+1}$ on the right-hand side and in $\R^{n}$ on the left-hand side of the equations.} 
    Thus, the tangent cone $\Tan{\overline{\bv}}$ of $\Xnb$ at $\overline{\bv}$ is generated as follows:
    \begin{enumerate}[label=(\alph*)]
        \item  
        If $b_{1} \ge 2$, or if $b_{1} = 1$ and $k \ge 2$, then
        $\Tan{\overline{\bv}}$ is generated by
        \begin{equation*}
            \setof{\ee_{\pi(j)} - \ee_{n+1} \sst 1 \le j \le k} 
            \,
            \cup 
            \,
            \setof{-\ee_{\pi(k+1)} + \ee_{n+1}}
            \,
            \cup 
            \,
            \setof{\ee_{\pi(j-1)} - \ee_{\pi(j)} \sst k+2 \le j \le n}
        \end{equation*}
        if $k \le n-1$, and by
        $
            \setof{\ee_{j} - \ee_{n+1} \sst 1 \le j \le n} 
        $
        if $k = n$.

        \item  
        If $b_{1} = 1$ and $k \in \setof{0, 1}$, then
        $\Tan{\overline{\bv}}$ is generated by
        \begin{equation*}
            \setof{\ee_{\pi(j-1)} - \ee_{\pi(j)} \sst 2 \le j \le n}
            \,
            \cup
            \,
            \setof{-\ee_{\pi(2)} + \ee_{n+1}}.
        \end{equation*}
    \end{enumerate}
    These generating vectors of the tangent cone are the facet-defining outer normals of the normal cone.  
    Thus, for example, in the case where $b_{1} \ge 2$ and $k \le n-1$, the normal cone at $\overline{\bv}$ is defined by the inequalities
    \begin{alignat*}{2}
        x_{\pi(j)}
        &\le
        x_{n+1} && \qquad \text{for $1 \le j \le k$}, \\ 
        x_{n+1}
        &\le
        x_{\pi(k+1)}, \\
        x_{\pi(j-1)} 
        &\le
        x_{\pi(j)} && \qquad \text{for $k+2 \le j \le n$}; 
    \end{alignat*}
    and the other cases are handled similarly.  
    These inequalities determine the covering relations given in the statement of \cref{lem: Vertex posets}.
\end{proof}

\begin{theorem}\label{thm:h-poly_of_Xn}
    The $h$-polynomial of $\Xn$ is
    \begin{equation*}
        h(\Xn;z)
        =
        \begin{dcases*}
            \widetilde{A}_n(z) - n z A_{n-1}(z) & if $b_1=1$, \\
            \widetilde{A}_n(z) & if $b_1 \ge 2$.
        \end{dcases*}  
    \end{equation*}
\end{theorem}

\begin{proof}
    Consider first the case in which $b_{1} \ge 2$.  
    From \cref{fig:vertex poset b1 ge 2} (corresponding to \cref{sublem: Vertex posets b1 ge 2}), we see that $\smash{Q_{\smash{\overline{\pi(\y_{k})}}}}$ contains $1 + \des(\pi(k+1), \dotsc, \pi(n))$ descents if $0 \le k \le n-1$, and $0$ descents if $k = n$. 
    Thus, enumerating vertices as in the proof of \cref{cor:number_vertices},
    \begin{align*}
        \sum_{\bv \in \Ver(\Xn)} z^{\des(Q_{\overline{\bv}})}
        &=
        \sum_{k=0}^{n-1} 
        \sum_{\substack{I \subseteq [n] \sst \\ \card{I} = n-k}}
        \sum_{\sigma \in \mathfrak{S}_{n-k}}
        z^{1 + \des(\sigma)}
        +
        1 
        =
        1
        +
        z
        \sum_{k=1}^{n} \binom{n}{k} 
        \sum_{\sigma \in \mathfrak{S}_{k}} 
        z^{\des(\sigma)}
        =
        \widetilde{A}_{n}(z).
    \end{align*}

    Now consider the case in which $b_{1} = 1$.  
    From \cref{fig:vertex poset b1 ge 2,fig:vertex poset b_1 = 1}, we see that $\smash{Q_{\smash{\overline{\pi(\y_{k})}}}}$ contains $1 +\des(\pi)$ descents if $k = 1$, $1 + \des(\pi(k+1), \dotsc,\pi(n))$ descents if $2 \le k \le n-1$, and $0$ descents if $k =n$.  
    Again enumerating vertices as in the proof of \cref{cor:number_vertices},
    \begin{align*}
        \sum_{\bv \in \Ver(\Xn)} z^{\des(Q_{\overline{\bv}})}
        &=
        \sum_{\pi \in \mathfrak{S}_{n}}
        z^{1 + \des(\pi)}
        +
        \sum_{k=2}^{n-1} 
        \sum_{\substack{I \subseteq [n] \sst \\ \card{I} = n-k}}
        \sum_{\sigma \in \mathfrak{S}_{n-k}}
        z^{1 + \des(\sigma)}
        +
        1 \\
        &=
        1
        +
        z
        \sum_{k=1}^{n-2} \binom{n}{k} 
        \sum_{\sigma \in \mathfrak{S}_{k}} 
        z^{\des(\sigma)}
        +
        z
        \sum_{\pi \in \mathfrak{S}_{n}}
        z^{\des(\pi)}
        =
        \widetilde{A}_n(z) - n z A_{n-1}(z). \qedhere
    \end{align*}
\end{proof}

Since the $h$-polynomial of $\Xn$ depends only on whether $b_1=1$ or $b_1\ge 2$, comparing with results on faces numbers of $\mathfrak{X}_n(a, b, b, \dotsc, b)$ in \cite[Prop. 3.14]{HanadaLentferVindas}, we immediately have the results on face numbers of $\Xn$ for all $\bb$.

\begin{corollary} \label{cor: f-vector of Xn}
    Let $f_k$ be the number of $k$-dimensional faces of $\Xn$ for
    $k=0,1,\ldots,n$.  
    Then,
    \begin{equation*}
        f_k
        =
        \begin{dcases*}
            \textstyle\sum_{j=0,j\neq 1}^{n-k}
            \binom{n}{j}(n-k-j)!\, S(n-j+1,n-k-j+1)
            & if $b_1 =1$,\\
            \textstyle\sum_{j=0}^{n-k}
            \binom{n}{j}(n-k-j)!\,S(n-j+1,n-k-j+1)
            & if $b_1\ge 2$,
        \end{dcases*}
    \end{equation*}
    where $S(n,k)$ is the \emph{Stirling number of the second kind}.
\end{corollary}


\section{A building set perspective}\label{sec:bulding_set}

In this section, we completely characterize the face structure and identify the combinatorial types of all $\bb$-parking-function polytopes.
In particular, we show that even though not every $\bb$-parking-function polytope is a nestohedron (see \cref{def:nestohedron}), up to combinatorial equivalence, every $\bb$-parking-function polytope is either the classical parking-function polytope $\PF_n$ (if $b_1 = 1$) or the stellohedron $\St_n$ (if $b_1 \ge 2$).

We consider the $\Y$-generalized permutahedron associated with a certain class of collections $\B$.
We recall some background on building sets, nestohedra, and nested set complexes from \cite{PRW, Pos}.

\begin{definition}[{\hspace{1sp}\cite[Definition 7.1]{Pos}}]\label{def: building set} %
    Let $E$ be a finite set.  
    A \emph{building set} on $E$ is a collection $\B$ of nonempty subsets of $E$ that satisfies the conditions:
    \begin{itemize}
        \item[(B1)]
        If $I, J \in \B$ and $I \cap J \neq \emptyset$, then $I \cup J
        \in \B$.

        \item[(B2)]
        The singleton $\setof{i}$ lies in $\B$ for all $i \in E$.
    \end{itemize}
    In this case, we write $\B_{\max} \subset \B$ for the collection of maximal subsets in $\B$ under inclusion.
\end{definition}

\begin{definition} \label{def:nestohedron}
Let $\B$ be a building set in $[n]$.
A \emph{nestohedron} $P_{\B}$ is a $\Y$-generalized permutahedron with positive coefficients $y_I$ only for $I\in \B$: $P_n^{\Y}(\{y_I\}) = \sum_{I\in \B}y_I\bigtriangleup_I$.
\end{definition}

We will relate nestohedra to certain complexes formed from building sets.

\begin{definition}[{\hspace{1sp}\cite[Definition 6.4]{PRW}}]
    Let $\B$ be a building set on a finite set $E$.  A \emph{nested set} of $\B$ is a subset $N \subseteq \B \setminus \B_{\max}$ that satisfies the following conditions:

    \begin{itemize}
        \item[(N1)]
        If $I, J\in N$, then either $I \subseteq J$ or $I \supseteq J$ or $I \cap J = \emptyset$.

        \item[(N2)]
        If $J_1,\ldots, J_k \in N$ are pairwise disjoint with $k \ge 2$, then $J_1 \cup \dotsb \cup J_k \notin \B$.
    \end{itemize}
\end{definition}

From this definition, it follows that a subset of a nested set is a nested set.  
Therefore, the collection of all nested sets of $\B$ forms an abstract simplicial complex with vertex set $\B$, called the \emph{nested set complex} $\bigtriangleup_\B$ of $\B$.  
It turns out that the nestohedron $P_{\B}$ is always a simple polytope, and the face lattice of $P_{\B}$ does not depend on the values of the positive parameters $y_I$, but only on $\B$, and this face lattice is dual to the face lattice of the nested set complex $\bigtriangleup_{\B}$; see \cite[Theorem 7.4]{Pos}.

\begin{example} \label{ex:nestohedron}
The collection
$\B=\left\{ \setof{1},\setof{2},\setof{3},\setof{4},\setof{1,2,4},\setof{1,3,4},\setof{2,3,4},\setof{1,2,3,4}\right\}$ is a building set on $[4]$.  
The corresponding nested sets and their dimensions as faces in $\bigtriangleup_\B$ are as follows:
\begin{equation*}
    \left\{\begin{array}{lc}
         \dim = -1 &\emptyset\\
         \hline
         \dim = 0  &\setof{1},\setof{2},\setof{3},\setof{4},\setof{124},\setof{134},\setof{234}\\
         \hline
         \dim = 1 &\setof{1,2},\setof{1,3},\setof{2,3},\setof{1,4},\setof{2,4},\setof{3,4},\setof{1,124},\setof{1,134},\\
         & \setof{2,124},\setof{2,234},\setof{3,134},\setof{3,234},
         \setof{4,124},\setof{4,134},\setof{4,234}\\
         \hline
         \dim = 2 &\setof{1,2,3},\setof{1,2,124},\setof{1,3,134},\setof{2,3,234},\setof{1,4,124},\\
         &\setof{1,4,134},\setof{2,4,124},\setof{2,4,234},\setof{3,4,134},\setof{3,4,234}
    \end{array}\right\}.
\end{equation*}
This building set's nestohedron is the simple polytope in \cref{fig:nestohedron}.  
Each of the vertices corresponds to a maximal nested set.  
In fact, each of the faces of the nestohedron corresponds to the nested set that is the intersection of the maximal nested sets that correspond to the vertices of the face.

\begin{figure}[ht]
\begin{center}
\begin{tikzpicture}[scale=.6]
    \scriptsize
    \node (1) at (2,0) {\tiny $\{3,4,234\}$};
    \node (2) at (7,0) {\tiny$\{3,4,134\}$};
    \node (3) at (0,1) {\tiny$\{2,3,234\}$};
    \node (4) at (9,1) {\tiny$\{1,3,134\}$};
    \node (5) at (0,3) {\tiny$\{2,4,234\}$}
;
    \node (6) at (9,3) {\tiny$\{1,4,134\}$};
    \node (7) at (2,6) {\tiny$\{2,4,124\}$};
    \node (8) at (7,6) {\tiny$\{1,4,124\}$};
    \node (9) at (4.5,7) {\tiny$\{1,2,124\}$};
    \node (10) at (4.5,3) {\tiny$\{1,2,3\}$};

    \draw (1)--(2)--(4)--(6)--(8)--(9)--(7)--(5)--(3)--(1);
    \draw (1)--(5);
    \draw (2)--(6);
    \draw (7)--(8);
    \draw [dashed] (3)--(10);
    \draw [dashed] (4)--(10);
    \draw [dashed] (9)--(10);

\end{tikzpicture}
\caption{The nestohedron from \cref{ex:nestohedron}.}
\label{fig:nestohedron}
\end{center}
\end{figure}
\end{example}

There are two special families of vectors $\bb$ for which $\Xnb$ is a translation by $\ee_{n+1}$ of a nestohedron whose face structure can be determined without much more work.

\begin{definition}
A \emph{stellohedron} $\St_n$ is a nestohedron for the building set
\[  \setof{\setof{1}, \setof{2}, \dotsc, \setof{n}}
        \cup \setof{S \cup \setof{n+1}}_{S \subseteq [n]} .\]
\end{definition}

\begin{proposition} \label{prop:TwoBuildings}
    Let $\bb\in\Z_{>0}^n$.
    Let $y_I$ ($I\subseteq [n+1]$) be the associated coefficients given in \cref{prop: y-expression}. 
    Assume that $y_I>0$ for all sets $I$ for which $\card{I} = 1$ and for all sets $I$ for which $n+1\in I$ and $\card{I}\ge 3$.
    \begin{enumerate}[label=(\alph*), ref={\thetheorem(\alph*)}]
        \item
        \label[subprop]{subprop: TwoBuildings: y_I > 1}%
        If $y_I > 0$ for all sets $I$ of the form $I = \setof{i, n+1}$ ($i \le n$), then the $\bb$-parking-function polytope $\Xn$ is combinatorially equivalent to the stellohedron $\mathsf{St}_n$.

        \item
        \label[subprop]{subprop: TwoBuildings: y_I = 1}%
        If $y_I = 0$ for all sets $I$ of the form $I=\setof{i, n+1}$ ($i\le n$), then $\Xn$ is combinatorially equivalent to the nestohedron for the building set
        \begin{equation*}
            \setof{\setof{1}, \setof{2}, \dotsc, \setof{n}}
            \cup
            \setof{S \cup \setof{n+1}}_{S \subseteq [n], \, \card{S} \neq 1}.
        \end{equation*}%
    \end{enumerate}
\end{proposition}

\begin{proof}
    For \cref{subprop: TwoBuildings: y_I > 1}, observe that $\Xn$ is in this case combinatorially equivalent to the following Mikowski sum:
    \begin{equation*}
        \Xn
        \cong
        \Xnb
        \cong
        \sum_{i=1}^n
        \bigtriangleup_{\setof{i}}
        +
        \sum_{S\subseteq [n], \, S \neq \emptyset}
        \bigtriangleup_{S \cup \setof{n+1}}.
    \end{equation*}
    Note that the indices of the standard simplices in the above Minkowski sum range over the collection $\{\{1\},\{2\},\ldots,\{n\}\}\cup\{S\cup\{n+1\}\}_{\emptyset\neq S\subseteq[n]}$.  
    If we add $\setof{n+1}$ to this collection, the building set that results corresponds to the stellohedron \cite[Sec 10.4]{PRW}.  
    Adding $\bigtriangleup_{\setof{n+1}}$ to the Minkowski sum only translates the polytope by $\ee_{n+1}$ and does not change the face structure.  
    Hence, we have
    \begin{equation*}
        \Xn
        \cong
        \sum_{i=1}^n 
        \bigtriangleup_{\setof{i}}
        +
        \sum_{S\subseteq [n], \, S \neq \emptyset}
        \bigtriangleup_{S \cup \setof{n+1}}
        +
        \bigtriangleup_{\setof{n+1}}
        \cong
        \mathsf{St}_n.        
    \end{equation*}

    For \cref{subprop: TwoBuildings: y_I = 1}, we have in this case that $b_{1} = 1$, and so, similarly to the previous case,
    \begin{align*}
        \Xn
        &\cong
        \Xnb  
        \cong
        \sum_{i=1}^n
        \bigtriangleup_{\setof{i}}
        +
        \sum_{S \subseteq [n], \, \card{S} \ge 2}
        \bigtriangleup_{S \cup \setof{n+1}} \\
        &\cong 
        \sum_{i=1}^n 
        \bigtriangleup_{\setof{i}}
        +
        \sum_{S\subseteq [n], \, \card{S} \ge 2}
        \bigtriangleup_{S \cup \setof{n+1}}
        +
        \bigtriangleup_{\setof{n+1}}.
    \end{align*}
    The indices of the standard simplices in the above Minkowski sum range over the collection
    \[
       \setof{\setof{1},\setof{2}, \dotsc, \setof{n}}
       \cup
       \setof{S \cup \setof{n+1}}_{S\subseteq[n], \, \card{S} \neq 1},
    \]
    which is also a building set.  
    Therefore, $\Xn$ is combinatorially equivalent to the nestohedron corresponding to this building set.
\end{proof}

We will see later in \cref{cor:combinatorial_equivalence} that the two face structures in \cref{prop:TwoBuildings} are the only combinatorial types that can occur on $\Xn$ for \emph{any} $\bb\in\mathbb{Z}_{>0}^n$, even if the corresponding $y_I$ do not satisfy the positivity conditions of \cref{prop:TwoBuildings}.  
We introduce the notation
\begin{align*}
    \B_{\St_n}
    &\deftobe
    \{\{1\},\{2\},\ldots,\{n\}\}\cup\{S\cup\{n+1\}\}_{S\subseteq[n]}, \\
    \B_{\PF_n} 
    &\deftobe 
    \{\{1\},\{2\},\ldots,\{n\}\}\cup\{S\cup\{n+1\}\}_{S\subseteq[n], 
    \, \card{S} \neq 1},
\end{align*}
for the two building sets in \cref{prop:TwoBuildings}.
We also introduce some terminology from \cite{BHMPW2020} to characterize the nested sets.  
We say a subset $I \subseteq [n]$ and a (nonempty) chain $\mathcal{F} = (A_1 \subsetneq \dotsb \subsetneq A_k)$ in the Boolean lattice $(2^{[n]}, \subseteq)$ are \defing{compatible}, denoted by $I \le \mathcal{F}$, if $I \subseteq \min(\mathcal{F}) = A_1$. 
If $\mathcal{F}$ is the empty chain~$\emptyset$, then we consider every subset $I$ to be compatible with $\emptyset$.  
The nested sets for $\B_{\St_n}$ and $\B_{\PF_n}$ are characterized in the next propositions.

\begin{proposition}[{\hspace{1 sp}\cite[Prop. 3.7]{Liao2022Codes}}]
    The nested sets for $\B_{\St_n}$ have form $\setof{\setof{i}}_{i \in I} \cup \setof{A \cup \setof{n+1}}_{A \in \mathcal{F}},$ where $I \subseteq [n]$ and $\mathcal{F} \subset 2^{[n]}-\{[n]\}$ is a chain such that $I\le\mathcal{F}$ is a compatible pair.
\end{proposition}

The building set and the collection of nested sets in \cref{ex:nestohedron} is the case of $n=3$ of the proposition below.

\begin{proposition} \label{prop:nestedPF}
    The nested set for $\B_{\PF_n}$ is of the form $\setof{\setof{i}}_{i \in I} \cup \setof{A}_{A \in \mathcal{F}}$, where $I \subseteq [n+1]$ and $\card{I}\le 2$ if $n+1 \in I$, and $\mathcal{F} \subset 2^{[n+1]} - \setof{[n+1]}$ is a chain such that $\card{\min(\mathcal{F})} \ge 3$ and $n+1 \in \min(\mathcal{F})$, and $I \le \mathcal{F}$ is a compatible pair.
\end{proposition}

\begin{proof}
    Let $N$ be a nested set for $\B_{\PF_n}$.  
    Let $I$ be the subset consisting of $i \in [n+1]$ for any singleton $\setof{i}\in N$ and $\mathcal{F}$ be the collection of all subsets in $N$ that are not a singleton.  
    By the definition of nested set for $\B_{\PF_n}$, we must have $n+1 \in A$ and $\card{A}\ge 3$ for any $A \in \mathcal{F}$.  
    Pick any $A_1, A_2 \in \mathcal{F}$, we always have $n+1 \in A_1 \cap A_2 \neq \emptyset$.  
    Hence one of $A_{1}$ or
    $A_{2}$ must be a subset of the other and therefore $\mathcal{F}$ is a chain.  
    For any $i \in I$, we must have $i \in A$ for all $A \in \mathcal{F}$; otherwise, there exists $A' \in \mathcal{F}$ such that $\setof{i}$ and $A'$ are disjoint subsets in $N$ with $A' \cup \setof{i} \in \B_{\PF_n}$, which violates the condition (N2) in the definition of nested sets.
    This implies $I \subseteq \min(\mathcal{F})$, so $I \le \mathcal{F}$.  
    If $n+1 \in I$ and $\card{I} \ge 3$, then the union of all singletons in $I$, which is $I$ itself, is in $\B_{PF_n}$, which again violates condition (N2).  
    Thus, we have $\card{I} \le 2$ if $n+1 \in I$.  
    Conversely, it is routine to check that such a collection $\setof{\setof{i}}_{i\in I} \cup \setof{A}_{A \in \mathcal{F}}$ satisfies the two axioms of the definition of nested sets.
\end{proof}

In order to understand the face lattice of $\Xn$ in terms of building sets, we now generalize the notation $U_{i}$ for certain facets of $\Xn$ introduced immediately before \cref{lem:facets containing a vertex} by setting
\begin{equation*}
    U_{I} \deftobe \Biggbraces{(x_{1}, \dotsc, x_{n}) \in \Xn \sst \sum_{i
    \in I} x_{i} = \sum_{i = n - \card{I} + 1} S_{i}},
\end{equation*}
for all nonempty $I \subseteq [n]$, except in the case $\card{I} = n-1$ if $b_{1} = 1$.  
Thus, $U_{I}$ is the facet of $\Xn$ defined by the \eqref{x-parking}~inequality for the subset $I$. 
In terms of this notation, the facets $U_j$ defined prior to \cref{lem:facets containing a vertex} are given by $U_j = U_{I_j}$, where $I_j \deftobe \setof{j, j+1, \dotsc, n}$.
Comparing the supporting hyperplanes of the facets of $\Xn$ in \cref{thm:inequality_description} and the two building sets in \cref{prop:TwoBuildings}, we immediately arrive to the following lemma.

\begin{lemma} \label{lem:facets-bijection}
    Let $\B \deftobe \B_{\PF_n}$ if $b_{1} = 1$, and let $\B \deftobe \B_{\St_n}$ if $b_{1} \ge 2$.  
    The facets of $\Xn$ are in bijection with the elements of $\mathcal{B}\setminus\mathcal{B}_{max}=\mathcal{B}\setminus\{[n+1]\}$.  
    If $b_1=1$, then the bijection is as follows:
    \begin{alignat*}{2}
        L_{i} 
        & \longleftrightarrow
        \setof{i} 
        && \qquad \text{for $1\le i\le n$}, \\
        U_{[n]\setminus S}
        & \longleftrightarrow
        S \cup \setof{n+1},
        && \qquad \text{for $S\subsetneq [n]$, $\card{S} \ne 1$}.
    \end{alignat*}
    If $b_1\ge 2$, then the bijection is as follows:
    \begin{alignat*}{2}
        L_{i} 
        & \longleftrightarrow
        \setof{i} 
        && \qquad \text{for $1\le i\le n$}, \\
        U_{[n]\setminus S}
        & \longleftrightarrow
        S \cup \setof{n+1},
        && \qquad \text{for $S\subsetneq [n]$}.
    \end{alignat*}
\end{lemma}

For any $B\in\mathcal{B}\setminus\{[n+1]\}$, we denote by $F_B$ the corresponding facet of $\Xn$ under the bijection in \cref{lem:facets-bijection}.

\begin{lemma} \label{lem:vertices-bijecion}
    Let $\B \deftobe \B_{\PF_n}$ if $b_{1} = 1$, and let $\B \deftobe 
    \B_{\St_n}$ if $b_{1} \ge 2$. 
    Then the map 
    \begin{equation*}
        N
        \longmapsto
        \bigcap_{B\in N} F_B
    \end{equation*}
    is a bijection between the maximal nested sets $N$ (\emph{i.e.}, the facets of $\bigtriangleup_{\B}$) and the vertices of $\Xn$.
\end{lemma}

\begin{proof}
    Consider $b_1=1$, the maximal nested sets for $\B_{\PF_n}$ are of the form $\{\{i\}\}_{i\in I}\cup\{A_1\subsetneq A_2\subsetneq\ldots\subsetneq A_{n-|I|}\}$ satisfying the condition of \cref{prop:nestedPF} and $|A_{i+1}|-|A_{i}|=1$ for all $i$.  
    If $n+1\notin I$, then the corresponding vertex can be obtained as a vector in $\mathbb{R}^n$ by placing $1$s in positions $I$, then placing $S_n$ in position $[n+1]-A_{n-|I|}$, placing $S_{n-1}$ in position $A_{n-|I|}-A_{n-|I|-1}$, and so on, in the final step placing $S_{|I|+1}$ in position $A_2-A_1$.  
    In this way, we fill out all $n$ entries of the vector in $\mathbb{R}^n$, and the resulting vector is the permutation of $\y_{\card{I}}$ which is a vertex by \cref{subthm: vertices of Xn}.  
    If $n+1\in I$, then $I=\{\{i\},\{n+1\}\}$ for some $i\in [n]$.  
    Then we obtain the vertex as a vector in $\mathbb{R}^n$ by placing $1=b_1=S_1$ in position $i$, placing $S_n$ in position $[n+1]-A_{n-2}$, placing $S_{n-1}$ in position $A_{n-2}-A_{n-3}$, etc., placing $S_3$ in position $A_2-A_1$, in the final step placing $S_2$ in position $A_1-\{i,n+1\}$.  
    In this case, the resulting vector is the permutation of $\y_1 = \y_0$ which is also a vertex by \cref{subthm: vertices of Xn}.  
    As a result, this gives a bijection between the maximal nested set for $\B_{\PF_n}$ and the vertices of $\Xn$ for $b_1=1$ via the intersections of facets corresponding to the maximal nested sets.

    For the case of $b_1\ge 2$, the proof is similar.
\end{proof}

\begin{remark}\label{rem:vertex_poset}
In both cases $b_1=1$ and $b_1\ge 2$, one can further consider the vertex posets $Q_v$ for every $v$ in the set of vertices $\Ver(\Xn)$ (see \cref{subsec:h-poly}) cooperating with the bijection in  \cref{lem:vertices-bijecion}. 
It gives the following commutative diagram
\[
\begin{tikzcd}
    \Ver(\mathfrak{X}_n(\bx)) \arrow[rr ] \arrow[rd] &         & \{\mbox{maximal nested sets for $\B$}\}\\
                                      &  \{Q_v: v\in \Ver(\mathfrak{X}_n(\bx))\} \arrow[ru].
\end{tikzcd}
\]
For example, the following are the correspondences between the weakly increasing vertices of $\Xn$ for $b_1=1$ and their vertex posets and the maximal nested sets:
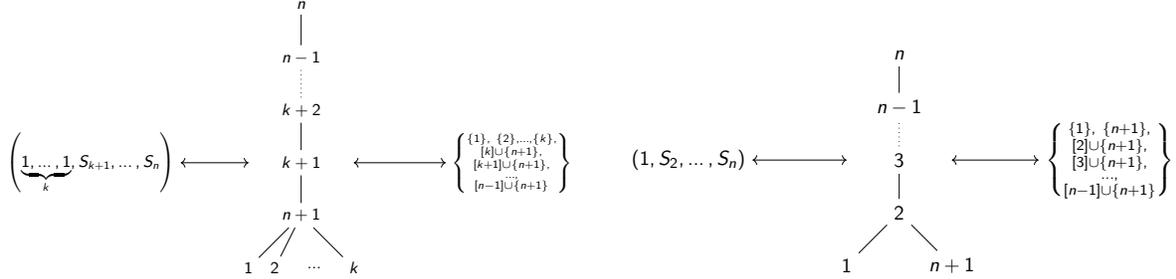
\begin{equation*}
    \scalebox{.7}{
        \begin{tikzpicture}[scale=1]
            \footnotesize
            \node (1) at (0,0) {$1$};
            \node (2) at (0.5,0) {$2$};
            \node (dot) at (1.25,0) {$\ldots$};
            \node (3) at (2,0) {$k$};
            \node (4) at (1,1) {$n+1$};
            \node (5) at (1,2) {$k+1$};
            \node (6) at (1,3) {$k+2$};
            \node (7) at (1,4) {$n-1$};
            \node (8) at (1,5) {$n$};

            \node (v) at (-3,2) {$\left(\underset{k}{\underbrace{1,\ldots,1}},S_{k+1},\ldots,S_n\right)$};
            \node (N) at (5,2) {
            $\left\{
            \substack{
            \{1\},\ \{2\},\ldots,\{k\},\\
            [k]\cup\{n+1\},\\ [k+1]\cup\{n+1\},\\ \ldots,\\ [n-1]\cup\{n+1\}
            }\right\}$ };

            \draw (1)--(4);
            \draw (2)--(4);
            \draw (3)--(4);
            \draw (4)--(5)--(6);
            \draw [dotted] (6)--(7);
            \draw (7)--(8);

            \draw [<->] (v)--(0,2);
            \draw [<->] (2,2)--(N);
        \end{tikzpicture}
        \quad \quad
        \begin{tikzpicture}[scale=1]
            \node (1) at (0,0) {$1$};
            \node (2) at (1,1) {$2$};
            \node (3) at (1,2) {$3$};
            \node (4) at (1,3) {$n-1$};
            \node (5) at (1,4) {$n$};
            \node (6) at (2,0) {$n+1$};

            \node (v) at (-3,2) {$(1,S_2,\ldots,S_n)$};
            \node (N) at (5,2) {
            $\left\{
            \substack{
            \{1\},\ \{n+1\},\\
            [2]\cup\{n+1\},\\ [3]\cup\{n+1\},\\ \ldots,\\ [n-1]\cup\{n+1\}
            }\right\}$ };

            \draw (1)--(2)--(3);
            \draw[dotted] (3)--(4);
            \draw (4)--(5);
            \draw (2)--(6);
            \draw [<->] (v)--(0,2);
            \draw [<->] (2,2)--(N);
        \end{tikzpicture}
.    }
\end{equation*}
\end{remark}

With the aid of Lemma \ref{lem:facets-bijection} and Lemma \ref{lem:vertices-bijecion}, we can completely characterize the face structure of the $\bb$-parking function polytope for any $\bb$ as follows.

\begin{theorem} \label{thm:face_structure_Xn}
    Let $\B \deftobe \B_{\PF_n}$ if $b_{1} = 1$, and let $\B \deftobe
    \B_{\St_n}$ if $b_{2} \ge 1$.  
    Then the map
    \begin{align*}
        \bigtriangleup_{\B} 
        \longrightarrow
        \{\text{nonempty faces of $\Xn$}\},
        \text{\qquad where \qquad} N 
        \longmapsto
        \mathcal{F}_N
        \deftobe
        \bigcap_{B\in N} F_B
    \end{align*}
    is an isomorphism between the dual of the face lattice of the nested set complex $\bigtriangleup_{\B}$ and the face lattice of $\Xn$.
\end{theorem}

\begin{proof}
    Consider $\bb=(b_1,\ldots,b_n)$ with $b_1=1$. 
    We first prove that the map from the nested sets in $\bigtriangleup_{\B_{\PF_n}}$ to the set of nonempty faces of $\Xn$ is well-defined.  
    From \cref{lem:facets-bijection}, the singleton nested set $\{B\}\in \bigtriangleup_{\B_{\PF_n}}$ corresponds to the facet $F_B$ of $\Xn$.  
    Since any intersection of facets is a face of the polytope, we know $\mathcal{F}_N$ is a face of $\Xn$ for any nested set $N\in \bigtriangleup_{\B_{\PF_n}}$.  
    For any nested set $N$, the face $\mathcal{F}_{N}$ is nonempty by \cref{lem:vertices-bijecion} because $\mathcal{F}_N$ contains those vertices that correspond to the maximal nested sets containing $N$.  
    Hence, the map is well-defined.

    Now we check the injectivity of the map.
    By \cref{subthm: Xn is simple}, $\Xn$ is a simple polytope. 
    In a simple polytope, every nonempty intersection of facets determines a unique face. 
    Hence, the injectivity of the map follows.

    We now show that the map is surjective.  
    Let $N'$ be a subset of $\B_{\PF_n}$ such that the face $\bigcap_{B\in N'}F_B$ is nonempty. 
    Say $N'=\{\{i\}\}_{i\in I'}\cup\{A\}_{A\in\mathcal{F}'}$ for some $I'\subset [n+1]$ and $\mathcal{F}'$ consisting of non-singleton subsets in $\B_{\PF_n}$, then
    \[
        \bigcap_{B\in N'}F_B=\left(\bigcap_{i\in I'}L_i\right)\cap\left(\bigcap_{A\in\mathcal{F}'}U_{[n]\setminus A}\right).
    \]
    We now show that $N'$ must be a nested set for $\B_{\PF_n}$ (Recall Proposition \ref{prop:nestedPF}).
 
\begin{itemize}[leftmargin=*]\setlength\itemsep{0.5em}
        \item By definition of $\B_{\PF_n}$, we have $n+1\in A$ and $|A|\ge 3$ for any $A\in\mathcal{F}'$ right away.
        
        \item First, for any $A\in\mathcal{F}'$ we must have $I'\subset A$; otherwise, there is $j\in I'$ but $j\notin A$ for some $A\in\mathcal{F}'$. 
        Then this face $\bigcap_{B\in N'}F_B$ is contained in the intersection of the facets $F_j \cap U_{[n]\setminus A}$, which is an empty set since $j\in [n]\setminus A$. 
        One can check the defining equations of the hyperplanes that there are no vertices of $\Xn$ in this intersection. 
        This contradicts the assumption that the face is nonempty.
        
        \item Second, we claim if $n+1\in I'$, then $|I'|\le 2$. 
        This is because the face $\bigcap_{B\in N'}F_B$ is contained in the facet $U_{[n]}$ whose supporting hyperplane is $\sum_{i=1}^n x_i=S_n+S_{n-1}+\ldots+S_1$. 
        Any vertex of $\Xn$ on $U_{[n]}$ is the permutation of $(S_1,S_2,\ldots,S_n)=(1,S_2,\ldots,S_n)$ which has exactly one coordinate being $1$. 
        Thus, $I'$ does not contain more than one element that is not $n+1$.
        
        \item In the final step, we show that $\mathcal{F}'$ must be a chain. 
        Let $A_1, A_2\in\mathcal{F}'$ with $|A_1|\le|A_2|$ then we must have $A_1\subseteq A_2$. 
        If not, there exists $j\in A_1$, but $j\notin A_2$. 
        Notice that any vertex $v=(v_1,\ldots,v_n)$ of $\Xn$ on the facet $U_{[n]\setminus A_2}$ has coordinates satisfying 
        \[
            \{v_i\}_{i\in[n]\setminus A_2}=\{S_n, S_{n-1},\ldots, S_{|A_2|+1}\};
        \] 
        while the vertex $v$ on $U_{[n]\setminus A_1}$ has coordinates satisfying 
        \[
            \{v_i\}_{i\in [n]\setminus A_1}=\{S_n,\ldots,S_{|A_2|+1},\ldots, S_{|A_1|+1}\}.
        \]
        Then there is no vertex of $\Xn$ that lies on $U_{[n]\setminus A_2}$ and $U_{[n]\setminus A_1}$ simultaneously since all the vertices have at most one coordinate equal to any one of $S_n, S_{n-1},\ldots, S_{|A_2|+1}$.
    \end{itemize}
Therefore, by \cref{prop:nestedPF} the subset $N'$ is a nested set for $\B_{\PF_n}$.

Now the map is a bijection. 
By the definition of the map, it is clear that any two nested sets $N_1\subseteq N_2$ if and only if $\mathcal{F}_{N_1}\supseteq\mathcal{F}_{N_2}$.
Therefore, the map is a poset isomorphism.
For the case that $b_1\ge 2$, the proof is similar.
\end{proof}

Recall that given a building set $\B\subset 2^{[n]}$, the face lattice of the nestohedron $P_{\B}=\sum_{I\in\B}y_I\bigtriangleup_I$ (where $y_I>0$) is completely determined by the building set $\B$ and is dual to the face lattice of the nested set complex $\bigtriangleup_{\B}$.
It is somewhat surprising that the face lattice of $\Xn$ is determined by the two building sets $\B_{PF_n}$ and $\B_{\St_n}$ because from Proposition \ref{prop: y-expression} we have seen that for most of $b$ the polytope $\Xnb$ is not a nestohedron (or a translation of one).
It will be interesting to understand how the value of $y_I$ affects the face structure of $\sum_{I}y_I\bigtriangleup_I$. 
More concretely, can we characterize the parameters $\{y_I\}$ such that $\Sigma y_I\bigtriangleup_I$ is combinatorially equivalent to $P_{\B}$ for a given building set $\B$?

From \cref{thm:face_structure_Xn}, we can explicitly tell which facets intersect at a given face and all the vertices on this face by observing the nested set corresponding to the face; we can see this exemplified next.  
\begin{example} 
For $\bb=(1,b_2,\ldots,b_6)$, consider the nested set $N=\{\{1\},\{3\}\}\cup (\{1,3,4,7\}\subsetneq\{1,2,3,4,7\})$.
Then the face $F_N$ is the intersection of the following facets of $\gpfp_6(\bb)$: 
\[
L_1:x_1=1, L_3:x_3=1, U_{\{1,3,4\}}:x_2+x_5+x_6=S_6+S_5+S_4, U_{\{1,2,3,4\}}:x_5+x_6=S_6+S_5.
\]
The face $F_N$ is $2$-dimensional and has $4$ vertices: \[(1,S_4,1,1,S_5,S_6), (1,S_4,1,S_3,S_5,S_6), (1,S_4,1,1,S_6,S_5), (1,S_4,1,S_3,S_6,S_5).\]
\end{example}

\begin{corollary}\label{cor:combinatorial_equivalence}
For \emph{any} $\bb=(b_1,b_2,\ldots,b_n)\in\mathbb{Z}_{>0}^n$, the $\bb$-parking-function polytope $\Xn$ is combinatorially equivalent to the classical parking-function polytope $\PF_n$ if $b_1=1$ or is combinatorially equivalent to the stellohedron $\St_n$ if $b_1\ge 2$.
\end{corollary}

Next, we provide a visualization of the two $3$-dimensional cases of \cref{cor:combinatorial_equivalence}.

\begin{example}\label{ex:comb_type}
The combinatorial type of $\Xn$ is as in \cref{fig:nestohedron} when $n=3$ and $b_1=1$ and is as in \cref{figure:St3} when $n=3$ and $b_1\ge 2$. 
The coordinate of each vertex can be obtained by using the bijection described in the proof of \cref{lem:vertices-bijecion}.   
\end{example}

\begin{figure}[ht]
\begin{center}
\scalebox{.5}{
\begin{tikzpicture}[scale=0.85]
    \tiny
    \node (1) at (3,0) {$\{3,34,234\}$};
    
    \node (3) at (0,1) {$\{2,3,234\}$};
    \node (5) at (0,4) {$\{2,24,234\}$};
    \node (C1) at (3,1.5) {$\{4,34,234\}$};
    \node (C2) at (2,3.5) {$\{4,24,234\}$};

    \node (2) at (7,0) {$\{3,34,134\}$};
    \node (C6) at (7,1.5) {$\{4,34,134\}$};
    \node (C5) at (8,3.5) {$\{4,14,134\}$};
    
    \node (6) at (10,4) {$\{1,14,134\}$};
    \node (4) at (10,1) {$\{1,3,134\}$};
    
    \node (7) at (1.5,6) {$\{2,24,124\}$};
    \node (8) at (8.5,6) {$\{1,14,124\}$};
    \node (C3) at (3.5,5.5) {$\{4,24,124\}$};
    \node (C4) at (6.5,5.5) {$\{4,14,124\}$};
    \node (9) at (5,7) {$\{1,2,124\}$};
    \node (10) at (5,3) {$\{1,2,3\}$};

    \draw (1)--(2)--(4)--(6)--(8)--(9)--(7)--(5)--(3)--(1);
    \draw (7)--(C3);
    \draw (C2)--(5);
    \draw (C4)--(8);
    \draw (1)--(C1)--(C2)--(C3)--(C4)--(C5);
    \draw (C1)--(C6);
    \draw (2)--(C6)--(C5)--(6);
    \draw [dashed] (3)--(10);
    \draw [dashed] (4)--(10);
    \draw [dashed] (9)--(10);
\end{tikzpicture}}   
\caption{The combinatorial type of $\Xn$ when $n=3$ and $b_1\ge 2$.}
\label{figure:St3}
\end{center}
\end{figure}


Up to this point we have observed $\Xn$ as a generalized permutahedron and have characterized its combinatorial types through the use of building sets.
In the next section, we exhibit $\Xn$ as a special polymatroid, which allows us to obtain further results.
 
\section{A polymatroid perspective}\label{sec:polymatroid_perspective}

Here we study the $\bb$-parking-function polytope through a polymatroid lens.
Polymatroids are a class of polyhedra that were introduced by Edmonds \cite{edmonds1970submodular} in 1970 as polyhedral generalizations of matroids.
We note that generalized permutahedra are equivalent (up to translation) to polymatroids. 
By exhibiting $\Xn$ as a polymatroid, we are able to draw conclusions about properties such as its combinatorial and circuit diameters. 
For background on diameters, see \cite{ykk-84,ks-10} and \cite{bfh-14, bdf-16, bv-17,env-22}.
We begin by recalling a few helpful terms and definitions from matroid theory \cite{oxley-06}. 

\begin{definition}
Let $f: 2^{[n]} \rightarrow \mathbb{R}$ be a set function. 
The function $f$ is said to be \emph{submodular} if for every $X, Y \subseteq [n]$ with $X \subseteq Y$ and every $z \in [n] \setminus Y$, the following inequality holds: $f(X \cup \{z\}) - f(X) \geq f(Y \cup \{z\}) - f(Y).$
The \defing{polymatroid} associated with $f$ is the polytope
\begin{equation*}
    P_f \deftobe \left\{x \in \mathbb{R}^{n} : x \geq 0, \ \sum_{i \in S} x_i \leq f(S) \text{ for all } S \subseteq [n] \right\}. 
\end{equation*}
\end{definition}
A set function $f \maps 2^{[n]} \to \R$ is said to be \defing{nondecreasing} if, for every $X, Y \subseteq [n]$ with $X \subseteq Y$, we have $f(X) \leq f(Y)$.
Henceforth, let $\fx \maps 2^{[n]} \to \Z_{+}$ be the set function defined by setting $\fx(\emptyset)\deftobe 0$ and 
\begin{equation*}
    \fx(I)
    \deftobe
    -\card{I} +
    \sum_{j=1}^n
    \min
    \setof{\card{I}, j} \, b_{n-j+1}, \qquad \text{ for } \emptyset \neq I \subseteq [n].
\end{equation*}

\begin{proposition}\label{prop:submodularity}
   The function $\fx$ is a nondecreasing submodular function.
\end{proposition}

\begin{proof}
By construction, $\fx$ is integer and nondecreasing. 
It remains to show that $\fx$ is submodular. 
Observe that $\fx(I)=-|I| + \sum_{j=1}^n \min \setof{\card{I}, j} \, b_{n-j+1}$ only depends on the cardinality $|I|$ of $I$, and not on the specific elements in $I$.  
It suffices to prove that $I \subset J \subsetneq [n] \Rightarrow  \fx(|I|+1) - \fx(|I|) \geq \fx(|J|+1) - \fx(|J|)$, or equivalently 
\[i\leq j < n \quad \Rightarrow \quad \fx(i+1) - \fx(i) \geq \fx(j+1) - \fx(j).\]
The difference $\fx(i+1) - \fx(i)$ is 
 \begin{align*}
        \fx(i+1) - \fx(i) = -(i+1)-(-i) + \sum_{k=1}^n
        (\min
        \setof{i+1, k} -\min
        \setof{i, k}) \, b_{n-k+1}= 
        -1 + \sum_{k=i+1}^n
         b_{n-k+1}.
\end{align*}
As $\bb=(b_1,\dots,b_n)\in \Z_{>0}^n$, \[\sum_{k=i+1}^n
         b_{n-k+1}\geq \sum_{k=j+1}^n
         b_{n-k+1} \quad\quad\qquad\text{for}\quad i \leq j.\] 
Thus, $\fx$ is submodular. 
\end{proof}

We translate the inequality description of $\Xn$ (\cref{subthm: facet-defining inequalities}) to a description of the form~$P_f$.
Setting $y_i \deftobe x_i - 1$, one obtains 
\begin{equation*}
    x_i \geq 1 
    \quad\Leftrightarrow\quad
    x_i -1 \geq 0
    \quad\Leftrightarrow\quad
    y_i\geq 0,
\end{equation*}
and the remaining constraints in \cref{subthm: facet-defining inequalities} take the form
\begin{equation*}
    \sum_{i \in I} x_i
    \le
    \sum_{j=1}^n \min \setof{\card{I}, j} \, b_{n-j+1}
    \quad\Leftrightarrow\quad 
    \sum_{i \in I} y_i 
    \leq
    -\card{I}
    +
    \sum_{j=1}^n \min \setof{\card{I}, j} \, b_{n-j+1}.
\end{equation*}
The resulting polymatroid is now described as
\begin{equation*}
    \Xnp
    \deftobe
    \setof{%
        y \in \R^{n}
        \sst
        \text{%
            $y \geq 0$,
            $\sum_{i \in I} y_i \leq \fx(I)$
            for all 
            $I \subseteq [n]$
        }%
    },%
\end{equation*}
which proves the following.

\begin{theorem}\label{thm:polymatroid}
The polytope $\Xnp$ is a polymatroid. 
In fact, $\Xnp$ is a special polymatroid for which upper bounds on the sum of entries in any set~$I$ depend only on the cardinality $|I|$.
\end{theorem}

The polytopes $\Xn$ and $\Xnp$ are combinatorially equivalent. 
Thus, we can leverage the theory of polymatroids to obtain results about $\Xn$.
In particular, we note that known results about polymatroids allow for alternate directions to recover \cref{thm:inequality_description}.
Topkis \cite{topkis1984adjacency} studied extreme point adjacency in polymatroids \cite[Corollary 5.4]{topkis1984adjacency}, and his work can be used to deduce \cref{subthm: Xn is simple}.
Shapley's observation on the structure of the vertices of a polymatroid \cite{shapley1971cores} can be used to obtain \cref{subthm: vertices of Xn}. 
Edmonds \cite{edmonds1970submodular, schrijver2003combinatorial} showed that we can optimize linear functions over polymatroids in strongly polynomial by certifying \emph{total dual integrality} through the greedy algorithm.
In turn, this yields another direction for proving the inequality description in \cref{subthm:
facet-defining inequalities}.

Next, using \cref{thm:polymatroid} and \cref{prop:edges}, we can carry over known upper bounds on the combinatorial diameters of polymatroids \cite{t-92} to $\Xn$, of the form $\min\{2n,\frac{1}{2}n(n-1)+1\}$. 
Recall that the \emph{combinatorial diameter} of a polytope is the diameter of its edge graph, that is, the maximum length of a shortest edge-path between two vertices. 
We now provide a direct proof of these upper bounds, which allows us to improve them for the case $b_1=1$ and see tightness in all cases.

\begin{theorem}\label{thm:boundtight} 
If $b_1>1$, the combinatorial diameter of $\Xn$ is exactly $\min\{2n,\frac{1}{2}n(n-1)+1\}$. 
If $b_1=1$, the combinatorial diameter is exactly $\min\{2(n-1),\frac{1}{2}n(n-1)\}$. 
\end{theorem}

\begin{proof}
Note that $2n \leq \frac{1}{2}n(n-1)+1$ and $2(n-1) < \frac{1}{2}n(n-1)$ if and only if $n\geq 5$. For $n\leq 4$, tightness of the bounds can be verified exhaustively. For example, for $n=3$, it is clear from Figure \ref{b123} that the combinatorial diameter of $\gpfp_3(1,2,3)$ and $ \gpfp_3(2,3,4)$ is $3$ and $4$, respectively. 

We provide an argument for the bounds $2n$ and $2(n-1)$. Let $\mathbf{w}_k = \pi(\y_k)$ be a vertex of $\Xn$ where $k < n$.
By \cref{prop:edges}, $\mathbf{w}_k$ is adjacent to $\mathbf{w}_{k+1} = \pi(\y_{k+1})$, which corresponds to reducing the $(k+1)$\textsuperscript{th} entry of $\y_k$ to 1.
Consider the path $P_k = (\mathbf{w}_k, \mathbf{w}_{k+1}, \dots, \mathbf{w}_{n-1}, \mathbf{1})$ on $\Xn$, where $P_k$ ends at $\mathbf{v}_n=\mathbf{w}_n=\mathbf{1}$. 
Because $P_k$ begins at $\mathbf{w}_k$, the length of $P_k$ is at most $n$ when $b_1>1$. 
Concatenating paths when necessary, we have a path of length at most $2n$ between any two vertices of $\Xn$ with $b_1 > 1$. 
A path of this length is realized when traveling between the vertices $\mathbf{w}_0 = (b_1, S_2, \dots, S_n)$ and $\pi'(\mathbf{w}_0) = (S_n, S_{n-1}, \dots, S_2, b_1)$. 
When $b_1 = 1$, $\mathbf{w}_0 = \mathbf{w}_1$, so $P_0$ has length at most $n-1$. 
A path of length $2(n-1)$ is realized when travelling between vertices $\mathbf{w}_1 = (1, S_2, \dots, S_n)$ and $\pi'(\mathbf{w}_1) = (S_n, S_{n-1}, \dots, S_2, 1)$.

It remains to show tightness. We exhibit that the paths constructed above are shortest paths between the given pairs of vertices. 
The bound $2n$ for a path between $\mathbf{w}_0$ and $\pi'(\mathbf{w}_0)$ is achieved through a combination of $n$ reductions and $n$ increases of entries $S_i$, and no swaps. A path that exclusively uses swaps requires at least $\frac{1}{2}n(n-1)$ steps, and so is not better. 
Let $S_j$ be the largest value that is reduced during an edge walk. 
It can be seen that any walk where $j=1$ has to perform the same number of swaps as one exclusively using swaps; any walk where $j=n-1$ has to perform at least two further steps (in addition to $n-1$ reductions and $n-1$ increases). Let now $2\leq j \leq n-2$. Then the number of swaps required to arrive at a vertex where the $j$ largest values $S_n,\dots,S_{n-j+1}$ are in the target positions is bounded below by $j\cdot (n-j)$. (For every $1\leq i_1 < i_2 \leq n$, it takes $i_2-i_1$ swaps to move $S_{i_2}$ to the position in which $S_{i_1}$ lies; in a best case, the target positions for $S_n,\dots,S_{n-j+1}$ start these swaps with values $S_{j},\dots,S_{1}$.) 
Combined with the at least $j$ reductions and $j$ increases, the path has length at least $2j+j(n-j) \geq 2j+2(n-2) \geq 2n$.

An analogous argument with $n-1$ in place of $n$ shows tightness of $2(n-1)$ for $b_1=1$. 
\end{proof}

Circuits and circuit diameters generalize the concepts of edges and combinatorial diameters. 
We refer the reader to \cite{bdf-16,bfh-14,dhl-15,env-22,r-69} for background. 
The set of \defing{elementary vectors} or \defing{circuits} $\mathcal{C}(P)$ of a polyhedron $P$ correspond to the inclusion-minimal dependence relations in the constraint matrices of a (minimal) representation. 
Note that the elementary vectors or circuits of a polymatroid are {\em not} the circuits of the underlying matroid. 
The elementary vectors include all edge directions that appear for varying right-hand sides of a representation, and so $\mathcal{C}(\Xn)$ contains a superset of the edge directions found in \cref{prop:edges}.

\begin{lemma}\label{lem:circuits}
The elementary vectors of $\Xn$ include $\pm \ee_i$ and $\ee_i - \ee_j$ for all $i,j \leq n$ with $i \neq j$.
\end{lemma}

\begin{proof}
It suffices to exhibit that for any vector of type $\pm \ee_i$ or $\ee_i - \ee_j$ for all $i,j \leq n$ with $i \neq j$, there exists an $\Xn$ in which that vector appears as an edge.
In fact, they all appear in the same polymatroid: a matroid polyhedron $\Xn$ for a uniform matroid with rank $\geq 2$; \emph{c.f.} \cite{t-92}. 
\end{proof}

The associated \defing{circuit walks} begin and end at vertices, and take maximal steps along elementary vectors instead of edges, in particular through the interior of a polyhedron.
The study of circuit diameters is motivated by the search for lower bounds on the combinatorial diameter and the efficiency of circuit augmentation schemes \cite{dhl-15}.
The \defing{circuit diameter} of a polytope is the maximum length of a shortest circuit walk between any two vertices. 
Of particular interest are polyhedral families where there is a significant gap between the combinatorial and circuit diameter \cite{bfh-16,kps-17}.
We will prove that $\bb$-parking-function polytopes is one of these families.

Unlike the more technical conditions for the adjacency of vertices found in \cref{prop:edges}, steps along the elementary vectors of \cref{lem:circuits} allow for more general changes to a $\bb$-parking function than classic matroid operations: {\em any} two entries can be swapped (a step in direction $\pm (\ee_i-\ee_j)$), even for non-bases and even including an entry that is $1$; this contrasts with the basis-exchange property of matroids where specific pairs of elements of two bases can be exchanged. 
Further, {\em any} entry can be decreased or increased (a step in direction $\pm \ee_i$), respectively. 
The steps along the elementary vectors have to be of maximal length while remaining feasible.
Recall from the proof of \cref{prop:submodularity} that $\Xn$ is a special polymatroid for which $\fx(I)$ only depends on the cardinality $|I|$.
Thus, any maximal step in direction $\pm (\ee_i-\ee_j)$ indeed is a swap of the corresponding entries. 
A maximal step in direction $- \ee_i$ corresponds to a decrease of an entry to $1$. 
Note that such a step starting at $\pi(\y_k)$ arrives at a non-vertex when decreasing entry $S_i$ to $1$ for $i>k+1$. 
A maximal step in direction $\ee_i$ corresponds to an increase of an entry to the largest $S_i$ that does not appear yet. 

We are now ready to prove a bound on the circuit diameter of $\Xn$.

\begin{theorem}\label{thm:circuitdiameter}
If $b_{1} = 1$, the circuit diameter of $\Xn$ is at most $n$. 
If $b_{1} \ge 2$, the circuit diameter is at most $n-1$.  
Moreover, these bounds are tight when restricted to the circuits of \cref{lem:circuits}.
\end{theorem}

\begin{proof}
Recall that each vertex of $\Xn$ is a permutation of a point of the form $\y_k$ for $0\leq k \leq n$.
If $b_1=1$, then $k\geq 1$. 
Consider vertices $\pi_1(\y_{k_1})$ and $\pi_2(\y_{k_2})$ for two permutations $\pi_1, \pi_2$ and $k_1 \leq k_2$. 
Note that $\pi_1(y_{k_1})$ and $\pi_2(y_{k_2})$ have at least $k_1$ entries $1$ and share the largest $n-k_2$ values of $y_{k_2}$ (permuted to different entries). 
Circuit walks are not reversible, so we will explain how to walk from $\pi_1(\y_{k_1})$ to $\pi_2(\y_{k_2})$ and how to walk from $\pi_2(\y_{k_2})$ to $\pi_1(\y_{k_1})$ in at most the claimed number of steps.

For the walk from $\pi_1(\y_{k_1})$ to $\pi_2(\y_{k_2})$, order the entries in descending order by value in $\pi_2(\y_{k_2})$. 
In the first $n-k_2$ iterations, a swap of two entries can guarantee that the largest $n-k_2$ entries match. 
Then follow $k_2-k_1$ iterations in which entries are decreased to $1$. 
The remaining $k_1$ entries are already identical to $1$. 
We obtain a circuit walk of at most $n-k_1$ steps.

For the walk from $\pi_2(\y_{k_2})$ to $\pi_1(\y_{k_1})$, order the entries in descending order by value in $\pi_1(\y_{k_1})$. 
In the first $n-k_2$ iterations, a swap of two entries can guarantee that the largest $n-k_2$ entries match. 
Then follow $k_2-k_1$ iterations in which a $1$-entry is increased to the next-largest entry in $\pi_1(\y_{k_1})$.
The remaining $k_1$ entries are already identical to $1$. 
Again, we obtain a circuit walk of at most $n-k_1$ steps.

It remains to show that the given bounds are tight if one is restricted to the elementary vectors in \cref{lem:circuits}. 
To this end, observe that any elementary vector can only increase one entry. 
If $b_1 > 1$, the case where $0=k_1 \leq k_2=n$ is possible, and all $n$ entries need to be increased in a circuit walk from $\pi_1(y_{k_1})$ to $\pi_2(y_{k_2})$. 
If $b_1 > 1$ then $k_1 \geq 1$. In the case where $1=k_1 \leq k_2=n$, $n-1$ entries need to be increased. This proves the claim.
\end{proof}

For all $n\geq 2$, the bounds of \cref{thm:circuitdiameter} are strictly lower than the bounds for the combinatorial diameter of \cref{thm:boundtight}. 
Note that we only worked with a subset of~$\mathcal{C}(\Xn)$. 
A complete characterization of the whole set $\mathcal{C}(\Xn)$ remains open. 
Due to this restriction, we constructed an {\em integral} circuit walk \cite{bv-17}, where each intermediate point is integral. 
We believe that our bounds are tight if restricted to integral walks and leave this as an open question.

\section{Conclusion - Further perspectives}\label{sec:revisit}

Recall that the classical parking-function polytope $\PF_n$ is the convex hull of all parking functions of length $n$ in $\R^n$.
As noted in \cite[Remark 2.10]{HanadaLentferVindas} and \cite[Remark 4.8]{BCC}, there are other alternative perspectives on $\PF_n$.
It can be viewed as:

\begin{itemize}[leftmargin=*]\setlength\itemsep{0.5em}

    \item a special case of the main object of study for this paper, the $\bb$-parking-function polytope when $\bb=(1,1,1,\dots,1)$,
    
    \item a particular \emph{partial permutahedron} $\mathcal{P}(n,n-1)$ as defined in \cite{HS} and further studied in \cite{BCC},

    \item a specific \emph{polytope of win vectors} as defined by Bartels et al. in \cite{BartelsMountWelsh} (when specializing to the complete graph, the work of Backman \cite[Theorem 4.5 and Corollary 4.6]{Backman} provides alternate routes for computing the volume and lattice-point enumerator of $\PF_n$), and

    \item essentially as the stochastic sandpile model for the complete graph \cite{Selig}.
\end{itemize}

In addition to these interpretations, as direct corollaries to \cref{thm: x-park_is_GPerm}, \cref{thm:face_structure_Xn}, and \cref{thm:polymatroid}, we can also add that (a linear projection of) $\PF_n$:
\begin{itemize}
    \item is a special generalized permutahedron,

    \item can be seen as a building set, and

    \item is a special type of polymatroid, respectively. 
\end{itemize}

We conclude by adding an interpretation of the classical parking-function polytope as a projection of a relaxed Birkhoff or assignment polytope, and an interpretation of $\bb$-parking-function polytopes as projections of relaxed partition polytopes.

\subsection{Relation of \texorpdfstring{$\PF_n$}{PF} to Birkhoff/assignment polytopes}\label{subsec:birkhoff}

Consider an assignment problem for $n$ cars and $n$ parking spots, represented with decision variables $x_{ij}\in \{0,1\}$ indicating car $i$ being assigned to spot $j$. 
The feasible set is a {\em Birkhoff polytope} or {\em assignment polytope} as follows \cite{br-74}:
    \begin{align*}
    \sum_{j=1}^{n} x_{ij}  &= 1 \;\; \text{ for all } 1\leq i \leq n\\
    \sum_{i=1}^{n} x_{ij}  &= 1 \;\; \text{ for all } 1\leq j \leq n\\
    x_{ij} &\geq 0 \;\; \text{ for all }1\leq i \leq n, \text{ } 1 \leq j \leq n.
\end{align*}

Due to the total unimodularity of the underlying constraint matrix and integrality of the right-hand sides, the vertices satisfy $x_{ij}\in \{0,1\}$.
In fact, they are in one-to-one correspondence to all possible assignments of cars to parking spots.
We prove that a linear projection of a mild relaxation of this polytope gives $\PF_n$.

\begin{theorem}\label{thm:birkhoff} 
The polytope $\PF_n$ is a linear projection of a relaxed Birkhoff/assignment polytope.
\end{theorem}

\begin{proof}
Note that either of the two types of equality constraints in the description of the assignment polytope, combined with the non-negativity constraints, implies an upper bound of $x_{ij}  \leq 1$. 
This allows an equivalent description of the form
    \begin{align*}
    \sum_{j=1}^{n} x_{ij}  &= 1 \;\; \text{ for all } 1\leq i \leq n\\
    \sum_{i=1}^{n} x_{ij}  &\geq 1 \;\; \text{ for all }  1\leq j \leq n\\
    x_{ij} &\geq 0 \;\; \text{ for all }  1\leq i \leq n, \text{ }   1\leq j \leq n.
\end{align*}
Informally, each car is assigned to precisely one spot, and each parking spot requires at least one car assigned to it, which is only possible if one retains a one-to-one assignment. 

We now perform a relaxation where we allow for the assignment of multiple cars to the same (preferred) parking spot, following the design of $\PF_n$: the $i$\textsuperscript{th} car to arrive picks from the first $i$ spots. 
To represent the corresponding conditions $\alpha'_i\leq i$ in a nondecreasing rearrangement of the entries, one can replace the set of constraints $\sum_{i=1}^{n} x_{ij} \geq 1$  by sums of the first $k$ constraints for $ 1\leq k \leq n$:
    \begin{align*}
    \sum_{j=1}^{n} x_{ij}  &= 1 \;\; \text{ for all }  1\leq i \leq n\\
    \sum_{j=1}^{k}\sum_{i=1}^{n} x_{ij}  &\geq k \; \,\text{ for all }  1\leq k \leq n\\
    x_{ij} &\geq 0 \;\; \text{ for all }  1\leq i \leq n, \text{ } 1\leq j \leq n.
\end{align*}
As $\sum_{i=1}^{n} x_{ij} \geq 1$ for all $ 1\leq j \leq n$ implies $\sum_{j=1}^{k}\sum_{i=1}^{n} x_{ij} \geq k$ for all $ 1\leq k \leq n$, this feasible set, which we call $\BF_n$, is a relaxation of the assignment polytope. 
Note that the associated constraint matrix lost total unimodularity. 
As part of our arguments below, we will prove that $\BF_n$ is integral, i.e., it has integral vertices.

To obtain $\PF_n$, consider the linear projection of $\BF_n$ using $(x_{ij}) \rightarrow (\sum_{j=1}^{n} j \cdot x_{ij})_{i=1,\dots,n}$, i.e., where the $i$\textsuperscript{th} entry is $\sum_{j=1}^{n} j \cdot x_{ij}$. 
For integral $(x_{ij})$, $x_{ij} = 1$ holds for exactly one $j$, the parking spot $j$ for car $i$. 
It remains to show that the projection of $\BF_n$ results in exactly $\PF_n$.
Recall that $\PF_n$ is an integral/lattice polytope in $\R^n$.
Thus, for $\mathbf{a}\in \PF_n$, which may have fractional components, there exists a convex combination $\mathbf{a}=\sum_{\ell=1}^n \lambda_\ell \mathbf{a}^\ell$ of integral parking functions $\mathbf{a}^\ell$. 
Let $\mathbf{x}^\ell=(x^\ell_{ij}) \in \{0,1\}^{n^2}$ be defined by $x^\ell_{ij}=1$ if and only if $a^\ell_{i}=j$. 
Then $\mathbf{x}=\sum_{\ell=1}^n \lambda_\ell \mathbf{x}^\ell$ lies in $\BF_n$ and projects onto $\mathbf{a}$.

Conversely, now let $\mathbf{x} \in \BF_n$. 
We have to prove that its projection $\mathbf{a}$ lies in $\PF_n$. 
To this end, we first prove that $\BF_n$ is integral. 
Let $\mathbf{x}$ be fractional and consider a bipartite graph $G$ with $V(G)=C\cup P$, $|C|= n = |P|$, and edges $(i,j)$ from $C$ to $P$ if and only if $0<x_{ij}<1$; informally, it is a representation of only the fractional assignments of parking spots to cars.
If there exists a cycle in the undirected graph underlying $G$, one can send flow along either orientation of that cycle and obtain a new feasible solution. 
In this case, $\mathbf{x}$ is not a vertex of $\BF_n$. 

If there does not exist a cycle in $G$, then there exist at least two nodes with degree $1$ in each connected component, as any tree has at least two leaves.
By $\sum_{j=1}^{n} x_{ij} = 1$ for all cars $i$, the degree of $i\in C$ in any component is at least two. 
Thus, there are at least two parking spots $j',j'' \in P$ with degree $1$, as well as a path between $j'$ and $j''$. 
Let $j'<j''$ and note that the constraints $\sum_{j=1}^{j'}\sum_{i=1}^{n} x_{ij}  \geq j'$ and $\sum_{j=1}^{j''}\sum_{i=1}^{n} x_{ij}  \geq j''$ must be satisfied strictly.
By sending (maximal) flow along the oriented path from $j'$ to $j''$, one obtains a new feasible solution with a strictly larger set of active constraints.
This again implies that $\mathbf{x}$ is not a vertex of $\BF_n$. 

We conclude that $\BF_n$ is a lattice/integral polytope with all vertices in $\{0,1\}^{n^2}$.
By the same arguments as above, $\mathbf{x}=\sum_{\ell=1}^t \lambda_\ell \mathbf{x}^\ell$ is a convex combination with $\mathbf{x}^\ell\in\{0,1\}^{n^2}$. 
Let $\mathbf{a}^\ell$ be the projection of $\mathbf{x}^\ell$ and observe it is an integral parking function. 
Thus, $\mathbf{x}$ projects to a convex combination $\mathbf{a}=\sum_{\ell=1}^n \lambda_\ell \mathbf{a}^\ell$ of integral parking functions $\mathbf{a}^\ell$. 
This proves the claim.
\end{proof}

\subsection{Relation of \texorpdfstring{$\Xn$}{$\Xn$} to partition polytopes}

Most of the proof for Theorem \ref{thm:birkhoff} transfers to $\bb$-parking-function polytopes $\Xn$. 
By using $x_{ij}\in \{0,1\}$ to indicate car $i$ being assigned a number $1 \leq j \leq S_n$, one can formulate the set of $\bb$-parking functions in the form
  \begin{align*}
    \sum_{j=1}^{S_n} x_{ij}  &= 1 \;\; \text{ for all }  1\leq i \leq n\\
    \sum_{j=1}^{S_k}\sum_{i=1}^{n} x_{ij}  &\geq k \; \,\text{ for all }  1\leq k \leq n\\
    x_{ij} &\geq 0 \;\; \text{ for all }  1\leq i \leq n, \text{ }  1\leq j \leq S_n.
\end{align*}
and the projection $(x_{ij}) \rightarrow (\sum_{j=1}^{n} j \cdot x_{ij})_{i=1,\dots,n}$ gives $\Xn$ by the same arguments.

A relation of this system to well-known polytopes is more involved than for $\PF_n$. 
The system can be seen as a relaxation of so-called \defing{(bounded-size) partition polytopes} \cite{bv-19a,bv-17}. 
Partition polytopes are special transportation polytopes and generalized Birkhoff or assignment polytopes. 
They are named for representing the (possibly fractional) assignment of $n$ items to $\ell$ clusters, where one specifies the total weight $\kappa_j$ assigned to each cluster $j$. 
A bounded-size partition polytope specifies lower and/or upper bounds on this weight instead of an exact weight. 
For our purposes, we set $\ell = S_n$ and only use lower bounds. 
Formally,
 \begin{align*}
    \sum_{j=1}^{S_n} x_{ij}  &= 1 \;\;\; \text{ for all }  1\leq i \leq n\\
    \sum_{i=1}^{n} x_{ij}  &\geq \kappa_j \;\; \text{ for all }  1\leq j \leq S_n\\
    x_{ij} &\geq 0 \;\;\;\, \text{ for all }  1\leq i \leq n, \text{ }  1\leq j \leq S_n.
\end{align*}
The first type of constraints guarantees that each item is assigned to a cluster (or partially to multiple clusters). 
The second type of constraints guarantees that each cluster receives a certain minimum weight.
Note that $S_n>n$ for any $\Xn$ that is not also a classical parking function $\PF_n$; informally, there are more clusters than items. 
For the existence of a feasible solution, it is necessary that at least some of the $\kappa_j$ satisfy $\kappa_j<1$. 

Let $S_{0}=0$ and $t\in [n]$, and set $\kappa_{j}=\frac{1}{S_{t}-S_{t-1}}$ for $S_{t-1} < j \leq S_{t}$. 
One can represent the conditions $\beta'_i\leq S_i$ in a nondecreasing rearrangement of the entries by replacing the set of constraints $\sum_{i=1}^{n} x_{ij} \geq \kappa_j$ by sums of the first $S_k$ constraints for $ 1\leq k \leq n$. 
This gives the above relaxed system. 

\begin{corollary}\label{cor:birkhoff} 
The polytope $\Xn$ is a linear projection of a relaxed bounded-size partition polytope.
\end{corollary}

We conclude by posing the following question to encourage further exploration of generalized parking function polytopes in view of possible relations to other well-known classes of polytopes. 
\begin{question}
 What is the relationship (if any) between $\bb$-parking function polytopes and partial permutahedra, polytopes of win vectors, or stochastic sandpile models for graphs?
\end{question}
 
\section*{Acknowledgements \& Funding}

This work was initiated at the 2023 Graduate Research Workshop in
Combinatorics, which was supported in part by the National
Science Foundation (NSF) under Award~1953985, and a generous award
from the Combinatorics Foundation. The authors thank the workshop
organizers and the University of Wyoming for fostering an excellent
research atmosphere.
We also thank Thomas Magnuson for conversations at the start of this
project and Martha Yip for helpful directions.

\begin{itemize}
    \item
    Steffen Borgwardt was supported by the Air Force Office
    of Scientific Research, Complex Networks, under
    Award~FA9550-21-1-0233.

    \item
    Steffen Borgwardt and Angela Morrison were supported by
    the NSF, CCF, Algorithmic Foundations, under Award~2006183.

    \item
    Danai Deligeorgaki is supported by the Wallenberg Autonomous Systems and Software Program; part of this research was performed while visiting the Institute for Mathematical and
    Statistical Innovation, which is supported by NSF Grant No.~DMS-1929348.

    \item
    Andr\'es R.~Vindas-Mel\'endez is supported by the NSF under Award~DMS-2102921.  
    This material is also based in part upon work supported by the NSF under Grant No.~DMS-1928930 and by the Alfred~P.~Sloan Foundation under grant G-2021-16778, while in residence at the Simons Laufer Mathematical Sciences Institute during Fall~2023.
\end{itemize}


\bibliographystyle{amsplain}
\bibliography{bibliography}


\end{document}